\documentclass[11pt,reqno,oneside]{amsproc}
\title[The Euler equations with variable coefficients]{The Euler equations with variable coefficients}
\author[B.~Ingimarson]{Benjamin Ingimarson}
\address{Department of Mathematics, University of Southern California, Los Angeles, CA 90089}
\email{ingimars@usc.edu}
\author[I.~Kukavica]{Igor Kukavica}
\address{Department of Mathematics, University of Southern California, Los Angeles, CA 90089}
\email{kukavica@usc.edu}
\author[A.~Tuffaha]{Amjad Tuffaha}
\address{Department of Mathematics and Statistics, American University of Sharjah, Sharjah, UAE}
\email{atufaha\char'100aus.edu}
  \chardef\forshowkeys=0
  \chardef\showllabel=0
  \chardef\refcheck=0
  \chardef\sketches=0
\usepackage{enumitem}
\usepackage{datetime}
\usepackage{fancyhdr}
\usepackage{comment}
\usepackage[T1]{fontenc}
\allowdisplaybreaks
\ifnum\forshowkeys=1

  \usepackage[notref,notcite,color]{showkeys}
\fi
\usepackage[margin=1in]{geometry}
\usepackage{amsmath, amsthm, amssymb}
\usepackage{times}
\usepackage{graphicx}
\usepackage[usenames,dvipsnames,svgnames,table]{xcolor}
\usepackage{marginnote}
\usepackage{soul}
\usepackage[unicode,breaklinks=true,colorlinks=true,linkcolor=black,urlcolor=black,citecolor=black]{hyperref}
\usepackage[most]{tcolorbox}
\usepackage{tikz}

\ifnum\refcheck=1
  \usepackage{refcheck}
\fi
\begin{document}
\def\YY{X}
\def\OO{\mathcal O}
\def\SS{\mathbb S}
\def\CC{\mathbb C}
\def\RR{\mathbb R}
\def\TT{\mathbb T}
\def\ZZ{\mathbb Z}
\def\HH{\mathbb H}
\def\RSZ{\mathcal R}
\def\LL{\mathcal L}
\def\SL{\LL^1}
\def\ZL{\LL^\infty}
\def\GG{\mathcal G}
\def\tt{\langle t\rangle}
\def\erf{\mathrm{Erf}}
\def\mgt#1{\textcolor{magenta}{#1}}
\def\ff{\rho}
\def\gg{G}
\def\sqrtnu{\sqrt{\nu}}
\def\ww{w}
\def\ft#1{#1_\xi}
\def\les{\lesssim}
\def\ges{\gtrsim}
\renewcommand*{\Re}{\ensuremath{\mathrm{{\mathbb R}e\,}}}
\renewcommand*{\Im}{\ensuremath{\mathrm{{\mathbb I}m\,}}}
\ifnum\showllabel=1
 \def\llabel#1{\marginnote{\color{lightgray}\rm\small(#1)}[-0.0cm]\notag}
\else
 \def\llabel#1{\notag}
\fi
\newcommand{\norm}[1]{\left\|#1\right\|}
\newcommand{\nnorm}[1]{\lVert #1\rVert}
\newcommand{\abs}[1]{\left|#1\right|}
\newcommand{\NORM}[1]{|\!|\!| #1|\!|\!|}
\newtheorem{theorem}{Theorem}[section]
\newtheorem{Theorem}{Theorem}[section]
\newtheorem{corollary}[theorem]{Corollary}
\newtheorem{Corollary}[theorem]{Corollary}
\newtheorem{proposition}[theorem]{Proposition}
\newtheorem{Proposition}[theorem]{Proposition}
\newtheorem{Lemma}[theorem]{Lemma}
\newtheorem{lemma}[theorem]{Lemma}
\theoremstyle{definition}
\newtheorem{definition}{Definition}[section]
\newtheorem{remark}[Theorem]{Remark}
\def\theequation{\thesection.\arabic{equation}}
\numberwithin{equation}{section}
\definecolor{mygray}{rgb}{.6,.6,.6}
\definecolor{myblue}{rgb}{9, 0, 1}
\definecolor{colorforkeys}{rgb}{1.0,0.0,0.0}
\newlength\mytemplen
\newsavebox\mytempbox
\makeatletter
\newcommand\mybluebox{%
    \@ifnextchar[
       {\@mybluebox}%
       {\@mybluebox[0pt]}}
\def\@mybluebox[#1]{%
    \@ifnextchar[
       {\@@mybluebox[#1]}%
       {\@@mybluebox[#1][0pt]}}
\def\@@mybluebox[#1][#2]#3{
    \sbox\mytempbox{#3}%
    \mytemplen\ht\mytempbox
    \advance\mytemplen #1\relax
    \ht\mytempbox\mytemplen
    \mytemplen\dp\mytempbox
    \advance\mytemplen #2\relax
    \dp\mytempbox\mytemplen
    \colorbox{myblue}{\hspace{1em}\usebox{\mytempbox}\hspace{1em}}}
\makeatother
\def\BMO{\text{BMO}}
\def\rr{r}
\def\weaks{\text{\,\,\,\,\,\,weakly-* in }}
\def\inn{\text{\,\,\,\,\,\,in }}
\def\cof{\mathop{\rm cof\,}\nolimits}
\def\Dn{\frac{\partial}{\partial N}}
\def\Dnn#1{\frac{\partial #1}{\partial N}}
\def\tdb{\tilde{b}}
\def\tda{b}
\def\qqq{u}
\def\lat{\Delta_2}
\def\biglinem{\vskip0.5truecm\par==========================\par\vskip0.5truecm}
\def\inon#1{\hbox{\ \ \ \ \ \ \ }\hbox{#1}}                
\def\onon#1{\inon{on~$#1$}}
\def\inin#1{\inon{in~$#1$}}
\def\FF{F}
\def\andand{\text{\indeq and\indeq}}
\def\ww{w(y)}
\def\ll{{\color{red}\ell}}
\def\ee{\epsilon_0}
\def\startnewsection#1#2{ \section{#1}\label{#2}\setcounter{equation}{0}}   
\def\nnewpage{ }
\def\sgn{\mathop{\rm sgn\,}\nolimits}    
\def\Tr{\mathop{\rm Tr}\nolimits}    
\def\div{\mathop{\rm div}\nolimits}
\def\curl{\mathop{\rm curl}\nolimits}
\def\dist{\mathop{\rm dist}\nolimits}  
\def\supp{\mathop{\rm supp}\nolimits}
\def\indeq{\quad{}}           
\def\period{.}                       
\def\semicolon{\,;}                  
\def\nts#1{{\cor #1\cob}}
\def\colr{\color{red}}
\def\colrr{\color{black}}
\def\colb{\color{black}}
\def\coly{\color{lightgray}}
\definecolor{colorgggg}{rgb}{0.1,0.5,0.3}
\definecolor{colorllll}{rgb}{0.0,0.7,0.0}
\definecolor{colorhhhh}{rgb}{0.3,0.75,0.4}
\definecolor{colorpppp}{rgb}{0.7,0.0,0.2}
\definecolor{coloroooo}{rgb}{0.45,0.0,0.0}
\definecolor{colorqqqq}{rgb}{0.1,0.7,0}
\def\colg{\color{colorgggg}}
\def\collg{\color{colorllll}}
\def\cole{\color{coloroooo}}
\def\coleo{\color{colorpppp}}
\def\cole{\color{black}}
\def\colu{\color{blue}}
\def\colc{\color{colorhhhh}}
\def\colW{\colb}   
\definecolor{coloraaaa}{rgb}{0.6,0.6,0.6}
\def\colw{\color{coloraaaa}}
\def\comma{ {\rm ,\qquad{}} }            
\def\commaone{ {\rm ,\quad{}} }          
\def\les{\lesssim}
\def\nts#1{{\color{blue}\hbox{\bf ~#1~}}} 
\def\ntsf#1{\footnote{\color{colorgggg}\hbox{#1}}} 
\def\blackdot{{\color{red}{\hskip-.0truecm\rule[-1mm]{4mm}{4mm}\hskip.2truecm}}\hskip-.3truecm}
\def\bluedot{{\color{blue}{\hskip-.0truecm\rule[-1mm]{4mm}{4mm}\hskip.2truecm}}\hskip-.3truecm}
\def\purpledot{{\color{colorpppp}{\hskip-.0truecm\rule[-1mm]{4mm}{4mm}\hskip.2truecm}}\hskip-.3truecm}
\def\greendot{{\color{colorgggg}{\hskip-.0truecm\rule[-1mm]{4mm}{4mm}\hskip.2truecm}}\hskip-.3truecm}
\def\cyandot{{\color{cyan}{\hskip-.0truecm\rule[-1mm]{4mm}{4mm}\hskip.2truecm}}\hskip-.3truecm}
\def\reddot{{\color{red}{\hskip-.0truecm\rule[-1mm]{4mm}{4mm}\hskip.2truecm}}\hskip-.3truecm}
\def\tdot{{\color{green}{\hskip-.0truecm\rule[-.5mm]{3mm}{3mm}\hskip.2truecm}}\hskip-.1truecm}
\def\gdot{\greendot}
\def\bdot{\bluedot}
\def\ydot{\cyandot}
\def\rdot{\cyandot}
\def\fractext#1#2{{#1}/{#2}}
\def\ii{\hat\imath}
\def\fei#1{\textcolor{blue}{#1}}
\def\vlad#1{\textcolor{cyan}{#1}}
\def\igor#1{\text{{\textcolor{colorqqqq}{#1}}}}
\def\igorf#1{\footnote{\text{{\textcolor{colorqqqq}{#1}}}}}
\newcommand{\p}{\partial}
\newcommand{\UE}{U^{\rm E}}
\newcommand{\PE}{P^{\rm E}}
\newcommand{\KP}{K_{\rm P}}
\newcommand{\uNS}{u^{\rm NS}}
\newcommand{\vNS}{v^{\rm NS}}
\newcommand{\pNS}{p^{\rm NS}}
\newcommand{\omegaNS}{\omega^{\rm NS}}
\newcommand{\uE}{u^{\rm E}}
\newcommand{\vE}{v^{\rm E}}
\newcommand{\pE}{p^{\rm E}}
\newcommand{\omegaE}{\omega^{\rm E}}
\newcommand{\ua}{u_{\rm   a}}
\newcommand{\va}{v_{\rm   a}}
\newcommand{\omegaa}{\omega_{\rm   a}}
\newcommand{\ue}{u_{\rm   e}}
\newcommand{\ve}{v_{\rm   e}}
\newcommand{\omegae}{\omega_{\rm e}}
\newcommand{\omegaeic}{\omega_{{\rm e}0}}
\newcommand{\ueic}{u_{{\rm   e}0}}
\newcommand{\veic}{v_{{\rm   e}0}}
\newcommand{\up}{u^{\rm P}}
\newcommand{\vp}{v^{\rm P}}
\newcommand{\tup}{{\tilde u}^{\rm P}}
\newcommand{\bvp}{{\bar v}^{\rm P}}
\newcommand{\omegap}{\omega^{\rm P}}
\newcommand{\tomegap}{\tilde \omega^{\rm P}}
\newcommand{\eps}{\varepsilon}
\newcommand{\weakto}{\rightharpoonup}
\renewcommand{\up}{u^{\rm P}}
\renewcommand{\vp}{v^{\rm P}}
\renewcommand{\omegap}{\Omega^{\rm P}}
\renewcommand{\tomegap}{\omega^{\rm P}}

\begin{abstract}
We establish local-in-time existence for the Euler equations on a bounded 
domain with space--time dependent variable coefficients, given initial data 
$v_0 \in H^r$ under the optimal regularity condition $r > 2.5$. 
In the case $r=3$, we further prove a Beale--Kato--Majda-type criterion 
that relates blow-up in the $H^r$ norm to the $\BMO$ norm 
of the variable vorticity $\zeta$.
\colb
\end{abstract}
\colb
\maketitle
\setcounter{tocdepth}{2} 
\tableofcontents

\startnewsection{Introduction}{sec00}
We consider the well-posedness and blow-up criteria for a system describing the Euler equations with space-time variable coefficients, 
  \begin{align}
   \begin{split} 
    &\partial_t v_i 
    + (v_m - \psi_m) a_{k m} \partial_k v_i
    + a_{k i} \partial_k q
    = 0
    \comma i =1,2,3,
    \\& 
    a_{j i} \partial_j v_i = 0
   \end{split}
   \label{EQaa}
  \end{align}
in $\Omega$, with the initial condition $v|_{t=0} = v_0$ and the boundary condition
  \begin{equation}
   v_k a_{j k} n_j
   = 
   \psi_k a_{j k} n_j
   \inon{on $\partial \Omega$ }
   ,
   \label{EQcc}
  \end{equation}
where $\Omega \subset \mathbb{R}^3$ is open and bounded with smooth boundary. 
In particular, we prove existence in the space $H^r$ where $r > 2.5$ (the minimal regularity for the classical Euler equations) for initial data $v_0 \in H^r$.
Along with a~priori estimates, we provide a Beale-Kato-Majda-type (BKM) criterion which relates the blow-up of a solution in the $H^r$ norm with the $\BMO$ norm of the variable vorticity $\zeta_i = \eps_{ijk} b_{\ell j} \partial_\ell v_k$.

The model under investigation arises when studying the incompressible Euler equations on domains with a free-moving boundary with a prescribed velocity via the Arbitrary Lagrangian Eulerian (ALE) change of variable, which transforms the problem to a fixed domain. The given function $\psi$ in this setting represents the velocity of the change of variables map that transforms the domain to a fixed one, and whose normal component on the boundary coincides with the normal velocity of the flow in what is known as the kinematic condition.
For example, this model arises naturally in the study of the free-boundary problem surrounding the fluid-structure interaction between an inviscid fluid with an elastic plate, also referred to as the Euler-plate system, recently studied in~\cite{KuT,AKT}. In this particular model, the fluid pressure acts as a dynamic forcing on a fourth-order equation governing the transverse motion of the boundary, while the kinematic condition matches the normal fluid velocity to the interface velocity. 
In the process of  constructing strong solutions for this particular model, a specific version of the above variable-coefficient Euler equations with inhomogeneous boundary conditions arises, where $\psi= \langle 0,0, \psi_{3}\rangle$ and $\psi_{3}$ is the harmonic extension of the interface velocity. 
Notably, the model~\eqref{EQaa}--\eqref{EQcc} also emerges when considering the classical Euler equations in a general time varying domain with a kinematic condition that matches the normal fluid velocity to the interface velocity. 
Finally, we note that our model reduces to the classical Euler equations when $a_{ij} = \delta_{ij}$ and $\psi = 0$. 


This paper is the first general treatment of well-posedness and blow-up of the incompressible variable coefficient Euler equations in a very general setting. The vast majority of mathematical treatises on the free boundary Euler equations pertain to the water wave model where pressure boundary conditions are imposed, see \cite{AKOT,ChL, CS1, KTV, KTVW, L, Lin, SZ1,SZ2, W1, W2, W3,WZZZ} and references therein. 
In the literature, we are only aware of an early work by Paolo Secchi \cite{Se} on the compressible Euler equations in a time-dependent domain evolving according to a prescribed smooth map, where local-in-time solutions are established for the variable coefficient system arising from the model. Our study is then motivated by the Euler-plate system describing the interaction between an inviscid flow and an elastic interface whose motion is governed by a fourth-order Kirchoff equation, 
first treated in~\cite{KuT},
where the authors considered the system in a periodic channel and provide
a~priori estimates for existence in spaces of minimal regularity
$H^{2.5+}$ for velocity
and constructed solutions
with the velocity in~$H^{3}$. The minimal case $r \in (2.5,3)$ was later treated in~\cite{AKT}.
In \cite{KuT,AKT}, the system was recast using the ALE (Arbitrary Lagrangian Eulerian) variables, which reduce to fluid equations of the form~\eqref{EQaa}--\eqref{EQcc}, where the Euler equations contain variable coefficients and inhomogeneous boundary data. 
In these two papers, the variable coefficients and boundary data tend to behave quite well; as mentioned before, they descend from a time-dependent change of variables via the harmonic extension of the boundary transversal displacement. 

In the present paper, we take much weaker assumptions on the variable-coefficient system, which complicate matters on well-posedness and blow-up. 
In addition, this is the first study of a Beale-Kato-Majda-type blow-up criterion for the Euler equations with variable coefficients and inhomogeneous boundary conditions, and the result in this general setting can be applied to the Euler-plate system, where no such result has been obtained in earlier works.

The paper is organized as follows.
In Section~\ref{sec01}, we state the model and main results.  
Theorem~\ref{T01} states the a~priori estimates for our local-in-time solution. 
Next, Theorem~\ref{T02} states the existence theorem. 
These two theorems are proven in the case $r \in (2.5,3]$, though the case $r > 3$ is identical.
Our final result, Theorem~\ref{T09a}, states a Beale-Kato-Majda-type criterion,
asserting that, under some extra assumptions on $a$ and
$\psi$, a blow-up in the $H^r$ norm necessitates a blow-up of $\|v\|_{H^1} + \|\zeta\|_{\BMO}$. 
Note that the classical BKM criterion was derived in
\cite{BKM} (see also \cite{P}), with the papers \cite{F,SY,Z} addressing the
case of a bounded domain.

In Section~\ref{sec02}, the a~priori estimates are attained through a div-curl-type bound on the fluid velocity. 
A considerable deviation from previous literature is that we are not assuming $a$ arises from a change of variables in which one could approximate it with the identity map for small time $[0,T]$.
Instead, we make a weaker assumption in that only $b = \cof(a^{-1})^T$ satisfies the Piola condition and $\det{a}$ is positively bounded above and below so that $a^T a$ is uniformly elliptic.
Note that the Piola condition is weaker than $a$ being conservative, the former being equivalent to the statement that the net flux of $b$ through any closed surface is $0$, yet it may not have vanishing curl.
With this, the original div-curl lemma as stated in~\cite{BB} does not pass muster. 
We get around this issue by proving a new variable div-curl lemma befitting our more general setting, which critically relies on the ellipticity of $b^T a$. 
Also, in sharp contrast with standard vorticity estimates in which the pressure drops out, our variable vorticity equation maintains a lower-order pressure term among other lower-order forcing terms propagated by the velocity.
Moreover, this differs from previous literature on the inviscid-plate model such as~\cite{KuT}, in which their variable vorticity equation is derived directly from the ALE change of variables and does not maintain the pressure let alone any forcing terms. 
Indeed, our model is not necessarily descended from an ALE change of coordinates, and so our vorticity equation is derived from applying the variable curl to~\eqref{EQaa}$_1$. 
Finally, in controlling the vorticity, we employ the strategy of extending the equation to $\mathbb{R}^3$ to handle the nonlinear terms with the Kato-Ponce inequality in the non-integer setting, $r \in (2.5,3)$.

In Section~\ref{sec03}, we construct a local-in-time solution for our system. 
In our approach, we first assume smoother variable coefficients and linearize~\eqref{EQaa}$_1$ into a transport equation on~$\mathbb{R}^3$.  
With the corresponding pressure, which is found by solving an elliptic problem with purely Neumann boundary conditions, we come to a solution for~\eqref{EQaa}$_1$ through an iterative procedure. 
Then, to derive the boundary condition~\eqref{EQcc}, we consider a 2D transport equation with tangential derivatives along the boundary of the domain. 
Surprisingly, our system under higher assumed regularity actually gives an inhomogeneous divergence condition which is proportional to the net flux of $b\psi$ through $\partial \Omega$, or equivalently, equal to its own average at each time~$t$. 
However, the solution under original regularity assumptions of the variable coefficients is attained through approximation under which the homogeneous divergence condition~\eqref{EQaa}$_2$ is attained due to a compatibility condition~\eqref{EQ05} assumed at the outset.
Unexpectedly, the solution to the original system with homogeneous variable divergence is the limit of solutions to systems of essentially the same structure but with inhomogeneous variable divergence.
Remarkably, the condition~\eqref{EQ05} is
intrinsic to the framework of the problem in the sense that
it is necessitated by the existence of a solution (see Remark~\ref{R01}). 
This reveals a relationship between $\psi$ and the incompressibility of the fluid. 
We end the section with a verification of compatibility of the elliptic Neumann problem for the pressure. 

In Section~\ref{sec04}, we prove a blow-up criterion for solutions in $H^3$ under higher regularity assumptions on $a$ and~$\psi$.
The non-integer case $r \in (2.5,3)$ seems unattainable with the method of extension we use in our vorticity estimates in Section~\ref{sec02}, which are only conducive to small-time estimates on $v$ in~$H^r$.
The same can be said when using the Sobolev-Slobodeckij norm where, in particular, trouble arises due to the forcing terms in the variable vorticity equation.
We note that with the absence of our forcing terms, the non-integer case would, in fact, be straightforward using the Sobolev-Slobodeckij norm. 
Proceeding in $H^3$, the core of the proof relies on controlling the stretching term $\|v\|_{W^{1,\infty}}$ with $\|v\|_{H^1} + \|\zeta\|_{\BMO}$.
To do this, we employ an estimate in the BMO norm on the elliptic system given in the proof of Lemma~\ref{L06}.
The presence of $\|v\|_{H^1}$ in the blow-up criterion is necessitated by the elliptic BMO estimate. 
Moreover, a significant trait of our system~\eqref{EQaa}--\eqref{EQcc} and major difference with standard inviscid problems is that the $L^2$ norm is not conserved, nor is it controlled solely by assumed data.  
A straightforward computation reveals the inhomogeneous boundary data to be the culprit as it maintains the pressure term. 
Therefore, the $L^2$ norm is still driven by the stretching term $\|\nabla v\|_{L^\infty }$; 
the same can be said for the $H^1$ norm.
As this dependence is inadequate in controlling $\|v\|_{W^{1,\infty}}$, we need to include $\|v\|_{H^1}$ in our blow-up criterion.

\startnewsection{Preliminaries and main results}{sec01}

Let $\Omega \subset \mathbb{R}^3$ be bounded with smooth boundary and $r \in (2.5,3]$.
Let $\psi \colon \Omega \times [0,T] \to \mathbb{R}^3$ and $a \colon \Omega \times [0,T] \to GL_3(\mathbb{R})$ be such that 
  \begin{equation}
   (\psi, \psi_t) \in L^\infty([0,T]; H^{r} \times H^{r-1})
   \label{EQ01}
  \end{equation}
and 
  \begin{equation}
   (a, a_t) \in L^\infty([0,T]; H^{r+1} \times H^{r-1}) 
   .
   \label{EQ02}
  \end{equation}
Note that all unindicated domains are $\Omega$.
Moreover, define 
  \begin{equation}
   J
   = 
   \det(a^{-1})
   \comma
   b 
   =
   \cof(a^{-1})^T
   ,
   \llabel{EQ03}
  \end{equation}
from which it follows 
  \begin{equation}
   a
   =
   \frac{1}{J} b
   .
   \label{EQ03a}
  \end{equation}
We assume that $b$ satisfies the Piola identity
  \begin{equation}
   \partial_j b_{j i}
   = 0 
   \comma
   i = 1,2,3
   .
   \llabel{EQ04}
  \end{equation}
We also assume the compatibility condition 
  \begin{equation}
   \int_{\partial \Omega} \partial_t (n_j b_{j i} \psi_i) 
   = 0
   \comma
   t \in [0,T]
   ,
   \label{EQ05}
  \end{equation}
which is also necessary for any solution to our system.
Finally, we assume that $J$ is uniformly bounded from above and below by positive constants on $\Omega \times [0,T]$. 
This is sufficient for uniform ellipticity of the matrix $b^{T}a$ due to~\eqref{EQ03a}.

Under these assumptions, we consider the Euler equations with variable coefficients
  \begin{align}
   \begin{split} 
    &\partial_t v_i 
    + (v_m - \psi_m) a_{k m} \partial_k v_i
    + a_{k i} \partial_k q
    = 0
    \comma i =1,2,3,
    \\& 
    a_{j i} \partial_j v_i = 0
    ,
   \end{split}
   \label{EQ06}
  \end{align}
in $\Omega$ with the initial condition
  \begin{equation}
   v|_{t = 0} 
   = 
   v_0
   \label{EQ07}
  \end{equation}
and the boundary condition
  \begin{equation}
   (v_k - \psi_k)a_{j k} n_j 
   = 0
   \inon{on $\partial \Omega$ }
   ,
   \label{EQ08}
  \end{equation}
where $n$ is the outward normal vector. 

The first theorem states the a~priori estimates for local existence of the above system. 
\begin{theorem}
\label{T01}
Assume $(v,q)$ is a smooth solution to~\eqref{EQ06}--\eqref{EQ08} on an interval $[0,T]$ with 
  \begin{equation}
   \int_\Omega q 
   = 0
   \llabel{EQ09}
  \end{equation}
such that
  \begin{equation}
   \|v_0\|_{H^r}
   \leq
   M
   \llabel{EQ10}
  \end{equation}
and 
  \begin{equation}
   \|a\|_{H^{r+1}},\,
   \|a_t\|_{H^{r-1}},\,
   \|\psi\|_{H^r},\,
   \|\psi_t\|_{H^{r-1}}
   \leq
   C
   \comma 
   t \in [0,T]
   ,
   \llabel{EQ12}
  \end{equation}
where $r \in (2.5, 3]$. Then $v$ and $q$ satisfy
  \begin{equation}
   \|v\|_{H^r},
   \|q\|_{H^r},
   \|v_t\|_{H^{r-1}}
   \leq
   C_0 M
   \comma
   t \in [0,\min\{T_0,T\}]
   ,
   \label{EQ11}
  \end{equation}
where $C_0 > 0$, and $T_0$ depends on~$M$.  
\end{theorem}
The second theorem states the corresponding existence result.
\begin{theorem}
\label{T02}
Assume that $v_0 \in H^r$, where $r \in (2.5,3)$, satisfies the divergence and boundary conditions stated above at time $t = 0$.
Further assume the compatibility condition~\eqref{EQ05}.
Then there exists a local-in-time solution $(v,q)$ to the system~\eqref{EQ06}--\eqref{EQ08} such that 
  \begin{align}
   \begin{split}
    &v \in L^\infty ([0,T]; H^r(\Omega)) \cap C([0,T]; H^{r-1}(\Omega))
    ,
    \\&
    v_t \in L^\infty([0,T]; H^{r-1}(\Omega))
    ,
    \\&
    q \in L^\infty ([0,T]; H^r(\Omega))
    ,
   \end{split}
   \llabel{EQa}
  \end{align}
for some time $T > 0$ depending on $v_0$ and assumed data $\psi$ and~$a$. 
Moreover, the solution  $(v,q)$ satisfies the estimate 
  \begin{equation}
   \|v(t)\|_{H^r}
   + \|\nabla q \|_{H^{r-1}}
   \les
   P(
     \|v_0\|_{H^r},
     \|a\|_{H^{r+1}},
     \|\psi\|_{H^r}
     )
   + \int_0^t P(
     \|\psi\|_{H^r},
     \|\psi_t\|_{H^{r-1}},
     \|a\|_{H^{r+1}},
     \|a_t\|_{H^{r-1}}
     )
   \,ds
   ,
   \llabel{EQb}
  \end{equation}
for $t \in [0,T]$.
\end{theorem}
Above and throughout the paper, the symbol $P$ denotes a non-negative generic polynomial of its arguments.

The third theorem states a Beale-Kato-Majda-type criterion for the system under $r = 3$,
for which we additionally assume
  \begin{equation}
   (\psi,\psi_t) \in L^\infty([0,\infty); H^{r+1} \times H^r)
   \label{EQ401a}
  \end{equation}
and 
  \begin{equation}
   (a,a_t) \in L^\infty([0,\infty); H^{r+2} \times H^r)
   .
   \label{EQ402a}
  \end{equation}
\begin{theorem}
\label{T09a}
Let $v$ be a solution to~\eqref{EQ06}--\eqref{EQ08} in the class $C([0,T]; H^r(\Omega))$ where $r = 3$. 
Suppose that $T = \hat{T}$ is the first time such that $v \notin C([0,T]; H^r(\Omega))$.
Then
  \begin{equation}
   \int_0^{\hat{T}} (\|v(t)\|_{H^1} + \|\zeta(t)\|_{\BMO})\, \mathrm dt
   =
   \infty
   .
   \llabel{EQ403a}
  \end{equation}
\end{theorem}
Note that a solution $v$ given from Theorem~\ref{T02} in the case $r =
3$ actually belongs to $C([0,T];H^3(\Omega))$ under the assumptions~\eqref{EQ401a} and~\eqref{EQ402a}. 
Indeed, for any $r >2.5$, a solution $v \in L^\infty([0,T]; H^r)$ satisfies $\partial_t v \in L^\infty ([0,T];H^{r-1})$ so that $v \in C_w([0,T];H^r)$. 
Hence, the continuity of $\|v(t)\|_{H^r}$ suffices.
Contrary to the div-curl approach to the estimates in Theorem~\ref{T01}, we can, when $r=3$, derive a differential inequality of the form
  \begin{equation}
   \frac{d}{dt} \|v\|_{H^3}
   \leq
   \|v\|_{H^3}^2
   .
   \llabel{EQzz}
  \end{equation}
The right hand side is $L^1$ in time and so $\|v(t)\|_{H^3}$ is absolutely continuous.

\startnewsection{Proof of the a~priori estimates}{sec02}
In this section, we prepare and give a proof of Theorem~\ref{T01}.
First, note that 
  \begin{equation}
   \|b\|_{H^{r+1}},
   \|J\|_{H^{r+1}} 
   \leq 
   P(
   \|a\|_{H^{r+1}}
   )
   \label{EQ13a}
  \end{equation}
and 
  \begin{equation}
   \|b_t\|_{H^{r-1}},
   \|J_t\|_{H^{r-1}}
   \leq
   P(
   \|a\|_{H^{r+1}},
   \|a_t\|_{H^{r-1}}
   )
   ,
   \label{EQ13b}
  \end{equation}
which can be seen by using the fractional chain rule. 

\subsection{Pressure estimates}
In this section, we prove the following estimate.
\begin{lemma}
\label{L02}
Under the assumptions of Theorem~\ref{T01}, we have 
  \begin{equation}
   \|q\|_{H^{r}}
   \leq
   P(
    \|v\|_{H^r},
    \|\psi\|_{H^r},
    \|\psi_t\|_{H^{r-1}},
    \|a\|_{H^{r+1}},
    \|a_t\|_{H^{r-1}}
   )
   .
   \label{EQ14}
  \end{equation}
\end{lemma}
First, we derive an elliptic Neumann problem for the pressure $q$ in the system~\eqref{EQ06}--\eqref{EQ08}. 
Applying the variable divergence $b_{j i} \partial_j(\cdot)_i$ to the velocity equation~\eqref{EQ06}$_1$, we get 
  \begin{equation}
   b_{j i} \partial_j \partial_t v_i 
   + b_{j i} \partial_j ((v_m - \psi_m) a_{k m} \partial_k v_i)
   + b_{j i} \partial_j (a_{k i} \partial_k q)
   = 0
   .
   \label{EQ15}
  \end{equation}
Due to the Piola condition $\partial_j b_{j i} = 0$ and divergence condition $b_{j i} \partial_j v_i = 0$, we have 
  \begin{equation}
   b_{j i} \partial_j (\partial_t v_i)
   =
   -\partial_j (\partial_t b_{j i} v_i)
   .
   \llabel{EQ16}
  \end{equation}
For the other terms in~\eqref{EQ15}, two applications of Piola deliver the elliptic equation
  \begin{equation}
   -\partial_j (b_{j i} a_{k i} \partial_k q)
   =
   - \partial_j (\partial_t b_{j i} v_i)
   + \partial_j(b_{j i}(v_m - \psi_m)a_{k m} \partial_k v_i)
   =: \partial_j \tilde{f}_j
   \inon{in $\Omega$.}
   \label{EQ17}
  \end{equation}
Next, we derive a Neumann condition for $q$ on the boundary. 
Testing~\eqref{EQ06} with $n_j b_{j i}$ and restricting to the boundary gives us
  \begin{equation}
   n_j b_{j i} \partial_t v_i
   + n_j b_{j i} (v_m - \psi_m) a_{k m} \partial_k v_i
   + n_j b_{j i} a_{k i} \partial_k q
   = 0
   .
   \label{EQ18}
  \end{equation}
We can eliminate the time derivative on $v$ through the boundary condition~\eqref{EQ08}.
Indeed, since $n$ is time-independent, we obtain
  \begin{align}
    n_j b_{j i} \partial_t v_i 
    =
    \partial_t(n_j b_{j i} v_i)
    - n_j \partial_t b_{j i} v_i
    =
    \partial_t(n_j b_{j i} \psi_i)
    - n_j \partial_t b_{j i} v_i
    .
   \llabel{EQ19}
  \end{align}
Hence, rearranging~\eqref{EQ18} gives us 
  \begin{equation}
   - n_j b_{j i} a_{k i} \partial_k q 
   =
   \partial_t (n_j b_{j i} \psi_i)
   - n_j \partial_t b_{j i} v_i
   + n_j b_{j i} (v_m - \psi_m) a_{k m} \partial_k v_i
   =: \tilde{g}
   \inon{on $\partial \Omega$.}
   \label{EQ20}
  \end{equation}
  Since $b^Ta$ is uniformly elliptic, the equations~\eqref{EQ17} and~\eqref{EQ20} constitute an elliptic boundary value problem with Neumann boundary conditions. 

Though we shall later reformat the equations to facilitate the proof of local existence, their current form better serves our computations for Lemma~\ref{L02}.

The following lemma contains the elliptic estimate on a Neumann problem of our type.
The proof is standard. 
\begin{lemma}
\label{L03}
Let $d \in C^2(\bar{\Omega})$ and $s \in [1,3]$. Suppose that $u \in H^s \cap C^\infty$ is a solution to the elliptic system
  \begin{align}
   \begin{split}
    &-\partial_j (d_{j k} \partial_k u) 
    = 
    \div{f}
    \,,
    \\&
    -n_j b_{j k} \partial_k u
    = 
    g
   \end{split}
   \llabel{EQ21}
  \end{align}
with $\int_\Omega u = 0$.
Then 
  \begin{equation}
   \|u\|_{H^s}
   \les
   \|f\|_{H^{s - 1}}
   + \|g\|_{H^{s - 3/2}(\partial \Omega)}
   ,
   \llabel{EQ22}
  \end{equation}
where the constant depends polynomially on $\|d\|_{C^2}$. 
\end{lemma}
Now, we are ready to prove Lemma~\ref{L02}. 
We use the multiplicative Sobolev inequalities
  \begin{equation}
   \|fg\|_{H^k} 
   \les
   \|f\|_{H^\ell}\|g\|_{H^m}
   ,
   \llabel{EQ40}
  \end{equation}
for $\ell,m \geq k$, where $\ell + m > 1.5 + k$, and 
  \begin{equation}
   \|fgh\|_{H^k}
   \les
   \|f\|_{H^\ell} \|g\|_{H^m} \|h\|_{H^n}
   ,
   \llabel{EQ41}
  \end{equation}
for $\ell,m,n \geq k$, where $\ell + m + n > 3 + k$. 

\begin{proof}[Proof of Lemma~\ref{L02}]
Since $b^Ta \in H^{r+1} \hookrightarrow C^2(\bar{\Omega})$ is uniformly elliptic, Lemma~\ref{L03} applies, and we have 
  \begin{equation}
   \|q\|_{H^{r}}
   \les
   \|\tilde{f}\|_{H^{r-1}} 
   + \|\tilde{g}\|_{H^{r-3/2}(\partial \Omega)}
   .
   \label{EQ42}
  \end{equation}
With $\tilde{f}$ and $\tilde{g}$ as in~\eqref{EQ17} and~\eqref{EQ20}, respectively, we write 
  \begin{align}
   \begin{split}
    \|\tilde{f}\|_{H^{r-1}} 
    &\les
    \|\partial_t b_{j i} v_i\|_{H^{r-1}}
    + \|b_{j i} (v_m - \psi_m) a_{k m} \partial_k v_i\|_{H^{r-1}}
    \\&\les
    \|b_t\|_{H^{r-1}}\|v\|_{H^{r-1}} 
    + \|b\|_{H^{r-1}}(\|v\|_{H^{r-1}} + \|\psi\|_{H^{r-1}})\|a\|_{H^{r-1}}\|v\|_{H^r}
    \\&\leq
    P(
    \|v\|_{H^r},
    \|\psi\|_{H^r},
    \|a\|_{H^{r+1}},
    \|a_t\|_{H^{r-1}}
    )
    ,
   \end{split}
   \label{EQ43}
  \end{align}
where we used \eqref{EQ13a}--\eqref{EQ13b} to bound $b$ in terms of~$a$. 
For the boundary term, we have 
  \begin{align}
   \begin{split}
    \|\tilde{g}\|_{H^{r-3/2}(\partial\Omega)}
    &\les
    \|\partial_t(n_jb_{j i} \psi_i)\|_{H^{r-3/2}(\partial\Omega)}
    + \|n_j \partial_t b_{j i} v_i\|_{H^{r-3/2}(\partial \Omega)}
    \\&\indeq
    + \|n_j b_{j i} (v_m - \psi_m)a_{k m} \partial_k v_i\|_{H^{r-3/2}(\partial \Omega)}
    \\&\les 
    \|\partial_t (b_{j i} \psi_i)\|_{H^{r-1}} 
    + \|\partial_t b_{j i} v_i\|_{H^{r-1}} 
    + \|b_{j i}(v_m - \psi_m) a_{k m} \partial_k v_i\|_{H^{r-1}}
    \\&\les
    P(
    \|v\|_{H^r},
    \|\psi\|_{H^r},
    \|\psi_t\|_{H^{r-1}},
    \|a\|_{H^{r+1}},
    \|a_t\|_{H^{r-1}}
    )
    .
   \end{split}
   \label{EQ44}
  \end{align}
Combining estimates~\eqref{EQ42}--\eqref{EQ44} completes the proof.
\end{proof}

When we construct a solution to~\eqref{EQ06}--\eqref{EQ08}, it will be advantageous to recast the elliptic problem~\eqref{EQ17} and~\eqref{EQ20} so that the order of derivative falling on $v$ is at most one order less than that of $q$ in their respective equations. 
With this in mind, we utilize the divergence and boundary conditions to make the following simplifications.

For the interior equation~\eqref{EQ17}, we can rewrite the third term as 
  \begin{align}
    \partial_j (b_{j i} (v_m - \psi_m)a_{k m} \partial_k v_i) 
    =
    b_{j i} \partial_j ((v_m - \psi_m) a_{k m} ) \partial_k v_i
    - (v_m -\psi_m)a_{k m} \partial_k b_{j i} \partial_j v_i
    ,
   \llabel{EQ45}
  \end{align}
where we used the product rule and the variable divergence condition $b_{j i} \partial_j v_i = 0$ 
Hence,~\eqref{EQ17} becomes
\begin{equation}
   -\partial_j (b_{j i} a_{k i} \partial_k q) 
   =
   -\partial_j (\partial_t b_{j i} v_i)
   + b_{j i} \partial_j ((v_m - \psi_m)a_{k m}) \partial_k v_i
   - (v_m -\psi_m)a_{k m} \partial_k b_{j i} \partial_j v_i
   \inon{in $\Omega$.}
   \label{EQ46}
  \end{equation}
For the boundary equation, we first write
  \begin{equation}
   \partial_k 
   = 
   (\delta_{k \ell} - n_k n_\ell)\partial_\ell 
   + n_kn_\ell \partial_\ell
   =
   \partial_k^\tau
   + n_k n_\ell \partial_\ell
   \llabel{EQ46a}
  \end{equation}
and change the fourth term in~\eqref{EQ20} to
  \begin{align}
   \begin{split}
    n_j b_{j i} (v_m - \psi_m) a_{k m} \partial_k v_i 
    &=
    (v_m - \psi_m) a_{k m} n_j b_{j i} \partial_k^\tau v_i
    + (v_m - \psi_m) a_{k m} n_kn_\ell n_j b_{j i} \partial_\ell v_i
    \\&=
    (v_m - \psi_m) a_{k m} \partial_k^\tau(n_j b_{j i} \psi_i) 
    - (v_m - \psi_m) a_{k m} \partial_k^\tau (n_j b_{j i}) v_i
    ,
   \end{split}
   \llabel{EQ47}
  \end{align}
where the second term in the first line vanishes by the boundary condition~\eqref{EQ08}, and the same boundary condition is substituted into the first term after using the product rule.
We obtain
  \begin{align}
   \begin{split}
    -n_j b_{j i} a_{k i} \partial_k q
    &=
    \partial_t (n_j b_{j i} \psi_i)
    - n_j \partial_t b_{j i} v_i
    + (v_m - \psi_m) a_{k m} \partial_k^\tau(n_j b_{j i} \psi_i)
    \\&\indeq
    - (v_m - \psi_m) a_{k m} \partial_k^\tau(n_j b_{j i}) v_i
   \end{split}
   \inon{on $\partial\Omega$.}
   \label{EQ48}
  \end{align}
The equations~\eqref{EQ46} and~\eqref{EQ48} are of the desired form. 

\subsection{Vorticity estimates}
We are now interested in analyzing the variable vorticity 
  \begin{equation}
   \zeta_i 
   =
   \eps_{ijk}b_{\ell j} \partial_\ell v_k
   ,
   \llabel{EQ49}
  \end{equation}
where $\eps_{ijk}$ denotes the Levi-Civita symbol.
First, we derive a variable vorticity equation for $\zeta$ by applying $\eps_{ijk}b_{\ell j}\partial_\ell(\cdot)_k$ to
  \begin{equation}
   \partial_t v_k 
   + (v_m - \psi_m) a_{r m} \partial_r v_k 
   + a_{r k} \partial_r q
   = 0
   \comma k= 1,2,3.
   \label{EQ50}
  \end{equation}
We handle the derivation term-by-term. 
Applying the variable curl to the first term in~\eqref{EQ50}, we get
  \begin{equation}
    \eps_{ijk} b_{\ell j} \partial_\ell (\partial_t v_k)
    =
    \partial_t\zeta_i 
    - \eps_{ijk} \partial_t b_{\ell j} \partial_\ell v_k
    \comma i = 1,2,3.
    \label{EQ51}
  \end{equation}
For the pressure term, we obtain
  \begin{align}
   \begin{split}
    \eps_{ijk} b_{\ell j}\partial_\ell (a_{rk} \partial_r q)
    &=
    \eps_{ijk} b_{\ell j} a_{r k} \partial_{\ell r} q 
    + \eps_{ijk} b_{\ell j} \partial_\ell a_{r k} \partial_r q
    \\&=
    J^{-1} \eps_{ijk} b_{\ell j} b_{r k} \partial_{\ell r} q
    + \eps_{ijk} b_{\ell j} \partial_\ell a_{r k} \partial_r q
    \\&=
    \eps_{ijk} b_{\ell j} \partial_\ell a_{r k} \partial_r q
    ,
   \end{split}
   \label{EQ52}
  \end{align}
where the first term in the second line vanishes due to symmetry.

The advection term is much more cumbersome. 
Applying the variable curl, we get 
  \begin{align}
   \begin{split}
    \eps_{ijk} b_{\ell j} \partial_\ell ((v_m - \psi_m) a_{r m} \partial_r v_k )
    &= 
    \eps_{ijk} b_{\ell j} \partial_\ell v_m a_{r m} \partial_r v_k
    - \eps_{ijk} b_{\ell j} \partial_\ell \psi_m a_{r m} \partial_r v_k
    \\& \indeq
    + (v_m - \psi_m) \eps_{ijk} b_{\ell j} \partial_\ell a_{r m} \partial_r v_k
    + (v_m - \psi_m) a_{r m} \eps_{ijk} b_{\ell j} \partial_{\ell r} v_k
    \\&=
    g_1^i + g_2^i + g_3^i + g_4^i
    .
   \end{split}
   \label{EQ53}
  \end{align}
We are particularly interested in $g_1$ and $g_4$, while the other terms are kept as forcing terms.
First, observe that
  \begin{equation}
   g_4^i
   =
   (v_m - \psi_m) a_{r m} \partial_r \zeta_i 
   - (v_m - \psi_m)a_{r m} \eps_{ijk} \partial_r b_{\ell j} \partial_\ell v_k
   .
   \label{EQ54}
  \end{equation}
Next, we claim that, in fact, $g_1^i = \zeta_j a_{p j} \partial_p v_i$. 
For the sake of brevity, we demonstrate this for~$i=1$.
Expanding the sum in $m$ for $J^{-1}g_i^i$, we obtain
  \begin{align}
   \begin{split}
    \eps_{1jk} b_{\ell j} \partial_\ell v_m b_{r m} \partial_r v_k
    &= 
    b_{\ell 2} \partial_\ell v_m b_{r m} \partial_r v_3
    - b_{\ell 3} \partial_\ell v_m b_{r m} \partial_r v_2
    \\&=
    (
    (b_{\ell 2} \partial_\ell v_1) (b_{r 1} \partial_r v_3)
    + (b_{\ell 2} \partial_\ell v_2) (b_{r 2} \partial_r v_3)
    + (b_{\ell 2} \partial_\ell v_3) (b_{r 3} \partial_r v_3)
    )
    \\& \indeq
    -(
    (b_{\ell 3} \partial_\ell v_1) (b_{r 1} \partial_r v_2)
    + (b_{\ell 3} \partial_\ell v_2) (b_{r 2} \partial_r v_2)
    + (b_{\ell 3}\partial_\ell v_3) (b_{r 3} \partial_r v_2)
    )
    .
   \end{split}
   \label{EQ55}
  \end{align}
We group together the first and fourth terms in the last two lines of~\eqref{EQ55}. Then we add and subtract $(b_{r 3}\partial_r v_1)(b_{\ell 2}\partial_\ell v_1)$, noting the ability to interchange the variables $\ell$ and~$r$, that is,
  \begin{align}
   \begin{split}
    &(b_{r 1} \partial_rv_3) (b_{\ell 2} \partial_\ell v_1)
    - (b_{r 1} \partial_r v_2) (b_{\ell 3} \partial_\ell v_1)
    \\&\indeq=
    (
     b_{r 1} \partial_r v_3
     - b_{r3}\partial_r v_1
    ) 
    b_{\ell 2} \partial_\ell v_1
    \\&\indeq\indeq
    -(
       b_{r 1} \partial_r v_2 - b_{r 2}\partial_r v_1
     )
    b_{\ell 3} \partial_\ell v_1
    \\&\indeq= 
    - (\eps_{2jk} b_{r j} \partial_r v_k)(b_{\ell 2} \partial_\ell v_1)
    - (\eps_{3jk} b_{r j} \partial_r v_k)(b_{\ell 3}\partial_\ell v_1)
    \\&\indeq=
    - \zeta_2 b_{p 2} \partial_p v_1
    - \zeta_3 b_{p 3} \partial_p v_1
    .
   \end{split}
   \label{EQ56}
  \end{align}
As we might expect, the remaining terms in the last two lines of~\eqref{EQ55} simplify to $-\zeta_1 b_{p 1}\partial_p v_1$. Indeed, 
  \begin{align}
   \begin{split}
    &(
    (b_{\ell 2} \partial_\ell v_2) (b_{r 2} \partial_r v_3)
    + (b_{\ell 2} \partial_\ell v_3) (b_{r 3} \partial_r v_3)
    )
    - (
    (b_{\ell 3} \partial_\ell v_2) (b_{r 2} \partial_r v_2)
    + (b_{\ell 3} \partial_\ell v_3) (b_{r 3} \partial_r v_2)
    )
    \\&\indeq\indeq=
    (b_{r 2}\partial_r v_3)(b_{\ell 2}\partial_\ell v_2)
    - (b_{\ell 3} \partial_\ell v_2)(b_{r 2}\partial_r v_2)
    + (b_{\ell 2} \partial_\ell v_3)(b_{r 3}\partial_r v_3)
    - (b_{r 3} \partial_r v_2)(b_{\ell 3} \partial_\ell v_3)
    \\&\indeq\indeq=
    (
    b_{r 2} \partial_r v_3 
    - b_{\ell 3} \partial_\ell v_2
    ) 
    b_{p2}\partial_p v_2
    + ( 
    b_{\ell 2}\partial_\ell v_3 
    - b_{r3}\partial_r v_2
    ) 
    b_{p3}\partial_p v_3
    \\&\indeq\indeq=
    ( 
    b_{r 2}\partial_r v_3 
    - b_{r 3}\partial_r v_2
    ) 
    (
    b_{p 2} \partial_p v_2 
    + b_{p 3} \partial_p v_3
    )
    \\&\indeq\indeq= 
    - (\eps_{1jk} b_{r j} \partial_r v_k) (b_{p_1}\partial_p v_1)
    \\&\indeq\indeq=
    -\zeta_1 b_{p 1} \partial_p v_1
    ,
   \end{split}
   \label{EQ57}
  \end{align}
where we used the divergence condition $b_{j i} \partial_j v_i = 0$ in the second-to-last line. 
These computations can be repeated for $i=2,3$ to get 
  \begin{equation}
   J^{-1}g_1^i 
   = 
   \eps_{ijk} b_{\ell j} \partial_\ell v_m b_{r m} \partial_r v_k
   =
   -\zeta_j b_{p j} \partial_p v_i
   .
   \label{EQ58}
  \end{equation}
Combining computations from~\eqref{EQ51}--\eqref{EQ54} and~\eqref{EQ56}--\eqref{EQ57}, we have that $\zeta$ satisfies
  \begin{equation}
   \partial_t \zeta_i
   + (v_m - \psi_m)a_{r m} \partial_r \zeta_i
   =
   \zeta_p a_{m p}\partial_m v_i
   + f_i
   \inon{in $\Omega$}
   \comma i = 1,2,3,
   \llabel{EQ59}
  \end{equation}
where we define 
  \begin{align}
   \begin{split}
    f_i 
    &=
    \eps_{ijk} \partial_t b_{\ell j} \partial_\ell v_k
    + (\eps_{ijk} b_{\ell j} \partial_\ell \psi_m) a_{r m} \partial_r v_k
    - (v_m - \psi_m) (\eps_{ijk} b_{\ell j} \partial_\ell a_{r m}) \partial_r v_k
    \\&\indeq
    - (v_m-\psi_m)a_{r m} (\eps_{ijk} \partial_r b_{\ell j} \partial_\ell v_k)
    - \eps_{ijk} b_{\ell j} (\partial_\ell a_{r k}) \partial_r q
    .
   \end{split}
   \label{EQ60}
  \end{align}
Notably, the forcing term $f_i$ contains a lower-order pressure term. 

Now, we turn to deriving estimates on $\zeta$ in~$H^{r-1}$.
We multiply~\eqref{EQ58} by $J$ to get 
  \begin{equation}
   J\partial_t \zeta_i
   + (v_m - \psi_m) b_{r m} \partial_r \zeta_i
   =
   \zeta_p b_{m p} \partial_m v_i
   + Jf_i
   .
   \label{EQ60a}
  \end{equation}
In order to perform non-tangential estimates in the Sobolev space $H^{r-1}(\Omega)$, we extend~\eqref{EQ60a} to~$\mathbb{R}^3$.  
Define the map $f \mapsto \tilde{f}$ as the $H^s(\Omega) \to H^s(\mathbb{R}^3)$ extension operator for $s \in [0,4]$ such that $\tilde{f}$ vanishes outside the $\Omega$ neighborhood $V_\alpha = \{x \in \mathbb{R}^3 \colon \dist( x,\Omega) < \alpha\}$ where $\alpha \geq 1$. 
Also, define $f\mapsto \bar{f}$ as the $H^s(\Omega) \to H^s(\mathbb{R}^3)$ extension operator for scalar valued functions such that $\bar{J} > c/2$ on $V_{\alpha}$ and vanishes outside $V_{2\alpha}$ where $c$ is such that $J > c$ for all $(x,t)\in\Omega \times [0,T]$.
This technical point shall be used in the proof for Lemma~\ref{L05}.

We now aim to obtain a~priori estimates for
  \begin{equation}
   \bar{J}\partial_t \theta_i 
   + (\tilde{v}_m - \tilde{\psi}_m)\tilde{b}_{r m}\partial_r \theta_i
   =
   \theta_p \tilde{b}_{m p} \partial_m \tilde{v}_i
   + \bar{J}\tilde{f}_i
   \inon{in $\mathbb{R}^3$}
   \comma i = 1,2,3.
   \label{EQ61}
  \end{equation}
\begin{lemma}
\label{L04}
A solution  $\theta \colon \mathbb{R}^3 \times [0,T] \to \mathbb{R}^3$ to~\eqref{EQ61} with $\theta(0) = \tilde{\zeta}(0)$ extends $\zeta$ on $\Omega \times [0,T]$. 
\end{lemma}
\begin{proof}
Define $\sigma = \zeta - \theta$. 
Then $\sigma$ satisfies 
  \begin{equation}
   J\partial_t \sigma_i
   + (v_m - \psi_m)b_{r m} \partial_r \sigma_i
   =
   \sigma_p b_{m p} \partial_m v_i
   \inon{in $\Omega$}
   \comma i=1,2,3.
   \llabel{EQ62}
  \end{equation}
We test with $\sigma$ in $L^2(\Omega)$, leading to
  \begin{align}
   \begin{split}
    \frac{1}{2} \partial_t \int_\Omega J |\sigma|^2 
    &=
    - \int_\Omega (v_m - \psi_m) b_{r m} \partial_r \sigma_i \sigma_i
    + \int_\Omega \sigma_p b_{m p} \partial_m v_i \sigma_i
    + \frac{1}{2} \int_{\Omega} \partial_t J |\sigma|^2
    \\&=
    I_1 + I_2 + I_3
    .
   \end{split}
   \llabel{EQ63}
  \end{align}
For $I_1$, integrating by parts gives 
  \begin{equation}
    I_1 
    = 
    \frac{1}{2}\int_\Omega \partial_r((v_m - \psi_m)b_{r m}) \sigma_i^2
    - \frac{1}{2} \int_{\partial \Omega} (v_m - \psi_m) b_{r m}n_r \sigma_i^2
    .
    \llabel{EQ64}
  \end{equation}
By the boundary condition~\eqref{EQ08}, the second integral vanishes, and since $H^r \hookrightarrow C^1(\bar{\Omega})$, 
  \begin{equation}
   |I_1|
   \les
   P(
     \|v\|_{H^r},
     \|\psi\|_{H^r},
     \|a\|_{H^{r+1}}
   )
   \|\sigma\|_{L^2}^2
   .
   \llabel{EQ65}
  \end{equation}
Note that $I_2$ and $I_3$ also have a similar bound due to embedding. 
Lastly, since $J$ is uniformly bounded below, we get
  \begin{equation}
   \partial_t \int_\Omega J |\sigma|^2
   \les
   P(
     \|v\|_{H^r},
     \|\psi\|_{H^r},
     \|a\|_{H^{r+1}},
     \|a_t\|_{H^{r-1}}
   )
   \int_\Omega J |\sigma|^2
   .
   \llabel{EQ66}
  \end{equation}
Since $\sigma(0) = 0$, applying Gronwall's lemma completes the proof.
\end{proof}

Next, we perform a~priori estimates for~$\theta$. 
Since $\bar{J}$ is uniformly bounded below, we have
\begin{equation}
   \|\zeta\|_{H^{r-1}(\Omega)}
   \les 
   \|\theta\|_{H^{r-1}(\mathbb{R}^3)}
   \les
   Y
   \llabel{EQ67}
  \end{equation}
and 
  \begin{equation}
   \|\zeta(0)\|_{H^{r-1}}
   \les 
   Y(0)
   \les
   \|\zeta(0)\|_{H^{r-1}}
   ,
   \llabel{EQ68}
  \end{equation}
where we define 
  \begin{equation}
   Y
   =
   \int_{\mathbb{R}^3} \bar{J}|\Lambda^{r-1}\theta|^2
   \llabel{EQ69}
  \end{equation}
with $\Lambda = (1-\Delta)^{1/2}$ on $\mathbb{R}^3 \times [0,T]$. 
Hence, to bound $\zeta$ by $\|v\|_{H^r}$ and assumed data for small time $[0,T_0]$, it suffices to prove the following lemma.
\begin{lemma}
\label{L05}
We have 
  \begin{equation}
   \partial_t Y
   \les
   P(
     \|v\|_{H^r},
     \|\psi\|_{H^r},
     \|\psi_t\|_{H^{r-1}},
     \|a\|_{H^{r+1}},
     \|a_t\|_{H^{r-1}}
   )
   (Y + Y^{1/2})
   .
   \llabel{EQ70}
  \end{equation}
\end{lemma}
\begin{proof}
All integrals below are taken over $\mathbb{R}^3$ unless indicated otherwise.
Applying $\Lambda^{r-1}$ to~\eqref{EQ61} and testing with $\Lambda^{r-1}\theta$ in $L^2(\mathbb{R}^3)$, we obtain 
  \begin{align}
   \begin{split}
    \frac{1}{2} \frac{dY}{dt} 
    &= 
    -\int \Big(
      \Lambda^{r-1}((\tilde{v}_m - \tilde{\psi}_m)\tilde{b}_{r m} \partial_r \theta_i) 
      - (\tilde{v}_m - \tilde{\psi}_m)\tilde{b}_{r m} \Lambda^{r-1}\partial_r \theta_i
    \Big) 
    \Lambda^{r-1}\theta_i
    \\&\indeq
    - \frac{1}{2}\int (\tilde{v}_m - \tilde{\psi}_m)\tilde{b}_{r m} \partial_r (\Lambda^{r-1}\theta_i)^2
    \\&\indeq
    + \int \Big(
      \Lambda^{r-1}(\theta_p \tilde{b}_{m p} \partial_m \tilde{v}_i)
      - \theta_p \tilde{b}_{m p} \Lambda^{r-1}\partial_m \tilde{v}_i
    \Big)
    \Lambda^{r-1}\theta_i
    \\&\indeq
    + \int \theta_p \tilde{b}_{m p}\partial_m \Lambda^{r-1}\tilde{v}_i \Lambda^{r-1}\theta_i
    \\&\indeq
    - \int \Big(
      \Lambda^{r-1}(\bar{J}\partial_t \theta_i)
      - \bar{J} \Lambda^{r-1}(\partial_t \theta_i)
    \Big)
    \Lambda^{r-1}\theta_i
    \\&\indeq
    + \frac{1}{2}\int \partial_t \bar{J}|\Lambda^{r-1}\theta |^2
    + \int \Lambda^{r-1}(\bar{J} \tilde{f}) \Lambda^{r-1}\theta_i
    \\&=
    I_1 + \cdots + I_7
    .
   \end{split}
   \label{EQ71}
  \end{align}
By~\cite[Theorem 1.9]{Li}, we get
  \begin{align}
   \begin{split}
    |I_1|
    &\les
    \Big(
      \|\Lambda^{r-1}((\tilde{v}_m - \tilde{\psi}_m)\tilde{b}_{r m})\|_{L^6} \|\Lambda \theta\|_{L^3}
      + \|\Lambda (\tilde{v}_m - \tilde{\psi}_m)\tilde{b}_{r m}\|_{L^\infty}\|\Lambda^{r-1}\theta\|_{L^2}
    \Big)
    \|\Lambda^{r-1}\theta\|_{L^2}
    \\&\les
    \Big(
      (\|\tilde{v}\|_{H^r} + \|\tilde{\psi}\|_{H^r})\|b\|_{H^{r+1}}\|\theta\|_{H^{3/2}}
      + (\|\tilde{v}\|_{H^{r-1}} + \|\tilde{\psi}\|_{H^{r-1}})\|\tilde{b}\|_{H^{r-1}}\|\Lambda^{r-1}\theta\|_{L^2}
    \Big)
    \|\Lambda^{r-1}\theta\|_{L^2}
    \\&\leq
    P(
      \|v\|_{H^r},
      \|\psi\|_{H^r},
      \|a\|_{H^{r+1}}
    )
    \|\Lambda^{r-1}\theta\|_{L^2}^2
    ,
   \end{split}
   \llabel{EQ72}
  \end{align}
where we used the embeddings $H^1 \hookrightarrow L^6$ and $H^{1/2} \hookrightarrow L^3$ as well as $H^{r-1}\hookrightarrow L^\infty$. 
For $I_2$, we integrate by parts and get 
  \begin{align}
   \begin{split}
    |I_2|
    &\les
    (\|\tilde{v}\|_{H^r} + \|\tilde{\psi}\|_{H^r})\|\tilde{b}\|_{H^{r+1}} \|\Lambda^{r-1}\theta\|_{L^2}^2
    \\&\leq
    P(
      \|v\|_{H^r},
      \|\psi\|_{H^r},
      \|a\|_{H^{r+1}}
    )
    \|\Lambda^{r-1}\theta\|_{L^2}^2
    .
   \end{split}
   \llabel{EQ73}
  \end{align}
For $I_3$, we follow a similar approach as for~$I_1$. We have
  \begin{align}
   \begin{split}
    I_3
    &\les
    \Big(
            \|\Lambda^{r-1}(\theta_p \tilde{b}_{m p})\|_{L^2} \|\Lambda \tilde{v}_i\|_{L^\infty}
            + \|\Lambda(\theta_p \tilde{b}_{m p})\|_{L^3} \|\Lambda^{r-1}\tilde{v}\|_{L^6}
    \Big)
    \|\Lambda^{r-1}\theta\|_{L^2}
    \\&\les
    \Big(
            \|\theta_p \tilde{b}_{p m}\|_{H^{r-1}} \|\Lambda \tilde{v}_i\|_{H^{r-1}}
            + \|\theta_p \tilde{b}_{p m}\|_{H^{r-1}} \|\tilde{v}\|_{H^r}
    \Big)
    \|\Lambda^{r-1}\theta\|_{L^2}
    \\&\leq
    P(
      \|v\|_{H^r},
      \|b\|_{H^{r+1}}
    )
    \|\Lambda^{r-1}\theta\|_{L^2}
    .
   \end{split}
   \llabel{EQ74}
  \end{align}
For $I_4$, we use the embedding $H^{r-1} \hookrightarrow L^\infty$ and apply H\"older's inequality to get
  \begin{align}
   \begin{split}
    |I_4|
    &\les
    \|\theta\|_{H^{r-1}} \|\tilde{b}\|_{H^{r+1}}\int |\partial_m \Lambda^{r-1}\tilde{v}_i| |\Lambda^{r-1}\theta_i|
    \\&\les
    \|\tilde{b}\|_{H^{r+1}} \|\tilde{v}\|_{H^r}\|\theta\|_{H^{r-1}}^2
    \leq
    P(
      \|v\|_{H^r},
      \|a\|_{H^{r+1}}
    )
    \|\Lambda^{r-1}\theta\|_{L^2}^2
    .
   \end{split}
   \llabel{EQ75}
  \end{align}
For $I_5$, we have 
  \begin{align}
   \begin{split}
   |I_5| 
   &\les
   \Big( 
           \|\Lambda^{r-1}\bar{J}\|_{L^6}\|\partial_t \theta\|_{L^3}
           + \|\Lambda \bar{J}\|_{L^\infty} \|\Lambda^{r-2}\partial_t \theta\|_{L^2}
   \Big)
   \|\Lambda^{r-1}\theta\|_{L^2}
   \\&\les
   \|J\|_{H^{r+1}}\|\Lambda^{r-2}\partial_t\theta\|_{L^2}\|\Lambda^{r-1}\theta\|_{L^2}
   \\&\leq
   P(
     \|v\|_{H^r},
     \|\psi\|_{H^r},
     \|\psi_t\|_{H^{r-1}},
     \|a\|_{H^{r+1}},
     \|a_t\|_{H^{r-1}}
   )
   (
    \|\Lambda^{r-1}\theta\|_{L^2}^2
    + \|\Lambda^{r-1}\theta\|_{L^2}
   )
   ,
   \end{split}
   \llabel{EQ76}
  \end{align}
which is clear once we have shown that 
  \begin{align}
   \begin{split}
    \|\Lambda^{r-2}\partial_t \theta\|_{L^2} 
    &\les
    P(
      \|v\|_{H^r},
      \|\psi\|_{H^r},
      \|a\|_{H^{r+1}}
    )
    \|\Lambda^{r-1}\theta\|_{L^2}
    \\&\indeq
    + P(
      \|v\|_{H^r},
      \|\psi\|_{H^r},
      \|\psi_t\|_{H^{r-1}},
      \|a\|_{H^{r+1}},
      \|a_t\|_{H^{r-1}}
    )
    .
   \end{split}
   \llabel{EQ77}
  \end{align}
To see this, take the $\|\Lambda^{r-2}\cdot \|_{L^2}$ norm of 
  \begin{equation}
   \partial_t \theta_i 
   + \bar{J}^{-1} (\tilde{v}_m 
   - \tilde{\psi}_m)\tilde{b}_{rm} \partial_r \theta_i
   =
   \bar{J}^{-1}\theta_p \tilde{b}_{m p} \partial_m \tilde{v}_i
   + \tilde{f}_i
   ,
   \label{EQ78}
  \end{equation}
and apply the fractional Leibniz rule.
The only term in~\eqref{EQ78} requiring special attention is~$\tilde{f}$.
Defined in~\eqref{EQ60}, we get the stronger inequality
  \begin{equation}
   \|\Lambda^{r-1}\tilde{f}\|_{L^2}
   \les
   P(
     \|v\|_{H^r},
     \|\psi\|_{H^r},
     \|\psi_t\|_{H^{r-1}},
     \|a\|_{H^{r+1}},
     \|a_t\|_{H^{r-1}}
   )
   ,
   \label{EQ80}
  \end{equation}
where we used the elliptic estimate~\eqref{EQ14}. 
Note that~\eqref{EQ78} is well-defined since $\bar{J}$ is positive on the support of all other terms.
For $I_6$, we immediately get 
  \begin{equation}
   |I_6|
   \les
   P(
     \|a\|_{H^{r+1}},
     \|a_t\|_{H^{r-1}}
   )
   \|\Lambda^{r-1}\theta\|_{L^2}^2
   .
   \llabel{EQ81}
  \end{equation}
Finally, using~\eqref{EQ80}, we have 
  \begin{equation}
   |I_7|
   \les
   P(
     \|v\|_{H^r},
     \|\psi\|_{H^r},
     \|\psi_t\|_{H^{r-1}},
     \|a\|_{H^{r+1}},
     \|a_t\|_{H^{r-1}}
   )
   \|\Lambda^{r-1}\theta\|_{L^2}
   .
   \llabel{EQ82}
  \end{equation}
The equation~\eqref{EQ71} and the estimates for $I_1$--$I_7$ complete the proof.
\end{proof}

\subsection{Velocity estimates}
Our method of acquiring estimates on $v$ critically relies on a variable div-curl lemma. 
\begin{lemma}
\label{L06}
For a fixed time $t \in [0,T]$, we have
  \begin{equation}
   \|v\|_{H^r}
   \les
   \|b_{j i} \partial_j v_i\|_{H^{r-1}}
   + \|\zeta\|_{H^{r-1}}
   + \|v_kb_{j k}n_j \|_{H^{r-1/2}(\partial\Omega)}
   + \|v\|_{L^2}
   ,
   \label{EQ83}
  \end{equation}
where the constant in the bound depends on a polynomial of~$\|b\|_{H^{r+1}}$.
\end{lemma}
\begin{proof}
Fix $t \in [0,T]$. 
Define $\varphi$ as the matrix with entries
  \begin{equation}
   \varphi_{kj} = b_{\ell j} \partial_\ell v_k - b_{\ell k} \partial_\ell v_j
   \comma k,j = 1,2,3
   .
   \label{EQ84}
  \end{equation}
Since $\|\varphi\|_{H^{r-1}} \leq C\|\zeta\|_{H^{r-1}}$, it suffices to prove~\eqref{EQ83} with $\zeta$ replaced by~$\varphi$.
We apply $b_{m k}\partial_m$ to~\eqref{EQ84} and after summing over $k$, we get
  \begin{equation}
   b_{m k} \partial_m \varphi_{kj} 
   = 
   b_{m k}\partial_m (b_{\ell j} \partial_\ell v_k) 
   - b_{m k} \partial_m (b_{\ell k} \partial_\ell v_j )
   .
   \label{EQ85}
  \end{equation}
We reformat~\eqref{EQ85} into a second order elliptic equation with Neumann boundary conditions, acquiring~\eqref{EQ83} through elliptic regularity. 

Indeed, applying Piola in the second term gives us 
  \begin{equation}
   b_{m k} \partial_m ( b_{\ell k} \partial_\ell v_j )
   =
   \partial_m (b_{m k} b_{\ell k} \partial_\ell v_j)
   .
   \llabel{EQ86}
  \end{equation}
For the first term, we again use Piola to shift terms and get 
  \begin{align}
   \begin{split}
    b_{m k} \partial_m \left(b_{\ell j} \partial_\ell v_k \right)
    &= 
    b_{m k}\partial_\ell\partial_m(b_{\ell j} v_k)
    \\&=
    \partial_\ell (b_{m k} \partial_m (b_{\ell j} v_k)) 
    - \partial_\ell b_{m k} \partial_m \left(b_{\ell j} v_k \right)
    \\&=
    \partial_\ell (b_{m k}b_{\ell j} \partial_m v_k)
    + \partial_\ell (b_{m k} \partial_m b_{\ell j} v_k)
    - \partial_\ell b_{m k} \partial_m (b_{\ell j} v_k)
    \\&=
    b_{\ell j} \partial_\ell (b_{m k} \partial_m v_k)
    + \partial_\ell(b_{m k} \partial_m b_{\ell j} v_k)
    - \partial_\ell b_{m k} \partial_m (b_{\ell j} v_k)   
    .
   \end{split}
   \llabel{EQ87}
  \end{align}
Hence, \eqref{EQ85} becomes
  \begin{equation}
   b_{m k} \partial_m \varphi_{kj}
   = 
   b_{\ell j} \partial_m(b_{m k} \partial_m v_k)
   + \partial_\ell (b_{m k} \partial_m b_{\ell j} v_k)
   - \partial_\ell b_{m k} \partial_m (b_{\ell j} v_k)
   - \partial_m(b_{m k} b_{\ell k} \partial_\ell v_j)
   .
   \llabel{EQ88}
  \end{equation}
Rearranging in a more suggestive way, we have 
  \begin{equation}
   -\partial_m(b_{m k} b_{\ell k} \partial_\ell v_j) 
   + \partial_\ell(b_{m k} \partial_m b_{\ell j} v_k) 
   - \partial_\ell b_{m k}\partial_m(b_{\ell k} v_k)
   =
   b_{m k}\partial_m \varphi_{kj}
   - b_{\ell j} \partial_\ell (b_{m k} \partial_m v_k)
   .
   \label{EQ89}
  \end{equation}
Note that the order of derivative falling on $b$ in the second term is still one due to Piola. 

Now, we construct a Dirichlet problem to aid in later estimates. 
Define $\nu \in C^\infty (\overline{\Omega}; \mathbb{R}^3)$ such that $\nu|_{\partial \Omega} = n$, where $n$ is the unit normal vector to~$\partial \Omega$.
Introduce the function
  \begin{equation}
   U 
   =
   v_p b_{i p} \nu_i
   \llabel{EQ90}
  \end{equation}
and the operator
  \begin{equation}
   L 
   = 
   -\partial_m (b_{m k} b_{\ell k} \partial_\ell (\cdot))
   .
   \llabel{EQ91}
  \end{equation}
By assumption, $L$ is uniformly elliptic. 
Since $U|_{\partial \Omega} = v_p b_{i p} n_i$, elliptic regularity gives us 
  \begin{equation}
   \|U\|_{H^r}
   \les_{\,b}
   \|LU\|_{H^{r-2}} 
   + \|v_p b_{i p} n_i\|_{H^{r-1/2}(\partial\Omega)}
   ,
   \llabel{EQ92}
  \end{equation}
where the inequality depends polynomially on~$\|b\|_{H^{r+1}}$.
To see this, note that $b \in H^{r+1} \hookrightarrow C^2(\overline{\Omega})$ with $r+1 > 3.5$.
The cases $r=2$ and $3$ are classical for~\eqref{EQ11} since $b$ and $\nabla b$ are at least $C^1$, and the embedding gives us the desired polynomial bound in~$\|b\|_{H^{r+1}}$. 
Then we interpolate between $r = 2$ and~$3$. 

To estimate $\|LU\|_{H^{r-2}}$, we expand 
  \begin{align}
   \begin{split}
    LU
    &=
    -\partial_m (b_{m k} b_{\ell k} \partial_\ell (v_p b_{i p}\nu_i))
    \\&= 
    -\partial_m(b_{m k}b_{\ell k} \partial_\ell v_p b_{i p} \nu_i)
    -\partial_m ( b_{m k} b_{\ell k} v_p \partial_\ell (b_{i p} \nu_i))
    \\&=
    (Lv_p)b_{i p}\nu_i
    - b_{m k} b_{\ell k} \partial_\ell v_p \partial_m (b_{i p} \nu_i)
    - \partial_m (b_{m k} b_{\ell k} v_p \partial_\ell(b_{i p} \nu_i))
    ,
   \end{split}
   \llabel{EQ93}
  \end{align}
giving us the estimate 
  \begin{align}
   \begin{split}
    \|U\|_{H^r}
    &\les
    \|(Lv_p)b_{i p} \nu_i\|_{H^{r-2}}
    + \|b_{m k} b_{\ell k} \partial_\ell v_p \partial_m (b_{i p} \nu_i)\|_{H^{r-2}}
    \\&\indeq 
    + \|\partial_m (b_{m k} b_{\ell k} v_p \partial_\ell (b_{i p} \nu_i) )\|_{H^{r-2}}
    + \|v_p b_{i p} n_i\|_{H^{r-1/2}(\partial \Omega)}
    \\&\les
    \|b\|_{H^{r+1}} \|Lv_p\|_{H^{r-2}}
    + \|b\|_{H^{r+1}}^3\|v\|_{H^{r-1}}
    + \|v_p b_{i p}n_i \|_{H^{r-1/2}(\partial \Omega)}
    \\&\les_{\, b}
    \|L v_p\|_{H^{r-2}}
    + \|v\|_{H^{r-1}}
    + \|v_p b_{i p} n_i\|_{H^{r-1/2}(\partial \Omega)}
    ,
   \end{split}
   \llabel{EQ94}
  \end{align}
where we used the multiplicative Sobolev inequalities. 
Next, we estimate $\|Lv_j\|_{H^{r-2}}$. From~\eqref{EQ89}, we get 
  \begin{align}
   \begin{split}
    \|Lv_j\|_{H^{r-2}}
    &\les
    \|b_{m k} \partial_m b_{\ell j} v_k\|_{H^{r-1}}
    + \|\partial_\ell b_{m k} \partial_m (b_{\ell j} v_k)\|_{H^{r-2}} 
    \\&\indeq
    + \|b_{m k} \partial_m \varphi_{kj}\|_{H^{r-2}}
    + \|b_{\ell j} \partial_\ell (b_{m k} \partial_m v_k)\|_{H^{r-2}}
    \\& \les_{\,b}
    \|v\|_{H^{r-1}}
    + \|\varphi\|_{H^{r-1}}
    + \|b_{m k} \partial_m v_k\|_{H^{r-1}}
    .
   \end{split}
   \llabel{EQ95}
  \end{align}
Hence,
  \begin{equation}
   \|U\|_{H^r}
   \les_{\, b}
   \|b_{m k} \partial_m v_k\|_{H^{r-1}}
   + \|\varphi\|_{H^{r-1}}
   + \|v_p b_{i p} n_i\|_{H^{r-1/2}(\partial \Omega)}
   + \|v\|_{H^{r-1}}
   .
   \label{EQ96}
  \end{equation}
Returning to the elliptic problem on $v_j$, we seek boundary estimates for the Neumann term $b_{m k} b_{\ell k} \partial_\ell v_j n_m$.
With this in mind, we define
  \begin{equation}
   W_{jk} 
   = 
   b_{m k} b_{\ell k} \partial_\ell v_j \nu_m
   \inon{in $\Omega$}   
   \comma j,k = 1,2,3
   .
   \llabel{EQ97}
  \end{equation}
For $j,k = 1,2,3$, we have
  \begin{equation}
   W_{jk} 
   = 
   b_{m k} \nu_m b_{\ell k} \partial_\ell v_j
   =
   b_{m k} \nu_m (b_{\ell j} \partial_\ell v_k - \varphi_{kj})
   =
   b_{m k} \nu_m b_{\ell j} \partial_\ell v_k
   - b_{m k} \nu_m \varphi_{kj}
   .
   \llabel{EQ98}
  \end{equation}
Letting $W_j = \sum_k W_{jk}$, we apply the product rule and find
  \begin{align}
   \begin{split}
    W_j 
    &=
    b_{\ell j} \partial_\ell (v_k b_{m k} \nu_m)
    - \partial_\ell (b_{m k} \nu_m) b_{\ell j} v_k 
    - b_{m k} \nu_m \varphi_{kj}
    \\&=
    b_{\ell j} \partial_\ell U
    - \partial_\ell (b_{m k} \nu_m) b_{\ell j} v_k
    - b_{m k} \nu_m \varphi_{kj}
    ,
   \end{split}
   \label{EQ99}
  \end{align}
so that $W_j \in H^{r-1}$. 
Restricting to the boundary, we obtain
  \begin{align}
   \begin{split}
    &\|b_{m k} b_{\ell k} \partial_\ell v_j n_m\|_{H^{r-3/2}(\partial \Omega)}
    \\&\indeq
    \les
    \|b_{\ell j} \partial_\ell U\|_{H^{r-3/2}(\partial \Omega)}
    + \|\partial_\ell (b_{m k} \nu_m) b_{\ell j} v_k\|_{H^{r-3/2}(\partial \Omega)}
    + \|b_{m k} \nu_m \varphi_{kj}\|_{H^{r-3/2}(\partial \Omega)}
    \\&\indeq\les 
    \|b_{\ell j} \partial_\ell U\|_{H^{r-1}}
    + \|\partial_\ell (b_{m k} \nu_m)b_{\ell j} v_k \|_{H^{r-1}}
    + \|b_{m k} \nu_m \varphi_{kj}\|_{H^{r-1}}
    \\&\indeq\les_{\, b}
    \|U\|_{H^r}
    + \|v\|_{H^{r-1}}
    + \|\varphi\|_{H^{r-1}}
    .
   \end{split}
   \llabel{EQ100}
  \end{align}
Applying elliptic regularity with~\eqref{EQ89} and the above estimate, we have
  \begin{align}
   \begin{split}
    \|\nabla v_j\|_{H^{r-1}}
    &\les_{\,b}
    \|b_{\ell j} \partial_\ell(b_{m k} \partial_m v_k)\|_{H^{r-2}}
    + \|b_{m k} \partial_m \varphi_{kj}\|_{H^{r-2}}
    + \|b_{m k} b_{\ell k} \partial_\ell v_j n_m\|_{H^{r-3/2}(\partial \Omega)} 
    \\&\les_{\,b} 
    \|b_{m k} \partial_m v_k \|_{H^{r-1}}
    + \|\varphi\|_{H^{r-1}}
    + \|U\|_{H^r}
    + \|v\|_{H^{r-1}}
    .
   \end{split}
   \llabel{EQ101}
  \end{align}
But with estimate~\eqref{EQ96}, this gives us 
  \begin{equation}
   \|\nabla v_j\|_{H^{r-1}}
   \les_{\, b}
   \|b_{m k} \partial_m v_k\|_{H^{r-1}}
   + \|\varphi\|_{H^{r-1}}
   + \|v_p b_{i p} n_i\|_{H^{r-1/2}(\partial \Omega)}
   + \|v\|_{H^{r-1}}
   .
   \llabel{EQ102}
  \end{equation}
Adding $\|v\|_{L^2}$ to both sides, we obtain
  \begin{equation}
   \|v\|_{H^{r}} 
   \les
   \|b_{m k} \partial_m v_k\|_{H^{r-1}}
   + \|\varphi\|_{H^{r-1}}
   + \|v_p b_{i p} n_i\|_{H^{r-1/2}(\partial \Omega)}
   + \|v\|_{H^{r-1}}
   .
   \label{EQ102a}
  \end{equation}
  Finally, to replace $\|\cdot\|_{H^{r-1}}$ with $\|\cdot\|_{L^2}$ on the right-hand-side of~\eqref{EQ102a}, we use the Sobolev interpolation and Young inequalities.
\end{proof}

We also have the following bound on $v_t$ which is verified directly from~\eqref{EQ06}.
\begin{lemma}
\label{L07}
Under the assumptions of Theorem~\ref{T01}, we have 
  \begin{equation}
   \|v_t\|_{H^{r-1}}
   \les
   P(
     \|v\|_{H^r},
     \|\psi\|_{H^r},
     \|\psi_t\|_{H^{r-1}},
     \|a\|_{H^{r+1}},
     \|a_t\|_{H^{r-1}}
   )
   .
   \label{EQ103}
  \end{equation}
\end{lemma}
We are now ready to prove Theorem~\ref{T01}. 
\begin{proof}[Proof of Theorem~\ref{T01}]
Using Lemma~\ref{L06}, we have 
  \begin{equation}
   \|v\|_{H^r}^2
   \les
   \|\zeta\|_{H^{r-1}}^2
   + \|\psi_k b_{j k} n_j\|_{H^{r-1/2}(\partial\Omega)}^2
   + \|v\|_{L^2}^2
   ,
   \label{EQ104}
  \end{equation}
where the divergence term vanishes by assumption and we used the boundary condition~\eqref{EQ08}.
From the vorticity bound given by Lemma~\ref{L05}, we have 
  \begin{align}
   \begin{split}
    \|\zeta\|_{H^{r-1}}^2 
    &\les 
    \|\zeta(0)\|_{H^{r-1}}^2
    + \int_0^t P(
             \|v\|_{H^r},
             \|\psi\|_{H^r},
             \|\psi_t\|_{H^{r-1}},
             \|a\|_{H^{r+1}},
             \|a_t\|_{H^{r-1}}
    )
    (Y + Y^{1/2})\, ds
    \\&\les
    \|b\|_{H^{r+1}}^2 \|v_0\|_{H^r}^2
    + \int_0^t P(
             \|v\|_{H^r},
             \|\psi\|_{H^r},
             \|\psi_t\|_{H^{r-1}},
             \|a\|_{H^{r+1}},
             \|a_t\|_{H^{r-1}}
    )
    \left(\frac{3}{2}Y + \frac{1}{2}\right) \, ds
    ,
   \end{split}
   \label{EQ105}
  \end{align}
where we used Young's inequality in the last line to linearize the integrand. 
Next, note that 
  \begin{align}
   \begin{split}
    \|v\|_{L^2}^2
    &\leq
    \|v_0\|_{L^2}^2
    + \int_0^t \|v_t\|_{L^2}^2\, ds
    \\&\les
    \|v_0\|_{L^2}^2
    + \int_0^t P(
             \|v\|_{H^r},
             \|\psi\|_{H^r},
             \|\psi_t\|_{H^{r-1}},
             \|a\|_{H^{r+1}},
             \|a_t\|_{H^{r-1}}
    )
    \, ds
    ,
   \end{split}
   \llabel{EQ106}
  \end{align}
employing the estimate~\eqref{EQ103}. 
Putting~\eqref{EQ105} and~\eqref{EQ103} into~\eqref{EQ104} and noting that 
  \begin{equation}
    \|\psi_k b_{j k} n_j\|_{H^{r-1/2}(\partial \Omega)}
    \les 
    \|b\|_{H^{r+1}}\|\psi\|_{H^r}
    ,
    \llabel{EQ107}
  \end{equation}
we get 
  \begin{align}
    \begin{split}
   \|v\|_{H^r}^2 
   &\les 
   \|v_0\|_{H^r}^2(1 + \|b\|_{H^{r+1}}^2)
   + \|b\|_{H^{r+1}}^2 \|\psi\|_{H^r}^2
    \\&\indeq
   + \int_0^t P(
             \|v\|_{H^r},
             \|\psi\|_{H^r},
             \|\psi_t\|_{H^{r-1}},
             \|a\|_{H^{r+1}},
             \|a_t\|_{H^{r-1}}
    )
    (Y+1)
    \, ds
    .
  \end{split}
   \label{EQ108}
  \end{align}
Applying Gronwall's lemma or a barrier argument to~\eqref{EQ105} and~\eqref{EQ108} completes the proof.
\end{proof}

\startnewsection{Existence}{sec03}
In this section we use the method in~\cite{KuT} to construct a solution to the system~\eqref{EQ06}--\eqref{EQ08}. 
Recall from Section~\ref{sec01} that the system is given by 
  \begin{align}
   &\partial_t v_i
   + (v_m - \psi_m)a_{k m} \partial_k v_i
   + a_{k i} \partial_k q
   = 0
   \inon{in $\Omega$}
   \comma i = 1,2,3\,,
   \label{EQ199}
   \\&
   a_{k i} \partial_k v_i
   = 0
   \label{EQ200}
  \end{align}
on a smooth bounded domain $\Omega \subset \mathbb{R}^3$, with the initial condition
  \begin{equation}
   v|_{t = 0}
   =
   v_0
   \label{EQ201}
  \end{equation}
and the boundary condition 
  \begin{equation}
   (v_k - \psi_k) a_{j k} n_j
   = 
   0
   \inon{on $\partial \Omega$}
   .
   \label{EQ202}
  \end{equation}
Again, we assume
  \begin{equation}
   (\psi, \psi_t) \in L^\infty ([0,T]; H^r \times H^{r-1})
   \label{EQ203}
  \end{equation}
and 
  \begin{equation}
   (a,a_t) \in L^\infty([0,T]; H^{r+1} \times H^{r-1})
   .
   \label{EQ204}
  \end{equation}
We define $b = \cof(a^{-1})^T$ and $J = \det(a^{-1})$, giving us $a = J^{-1}b$.
Assume that $b$ satisfies the Piola condition and $J$ is uniformly bounded above and below by positive constants. 
Finally, we assume 
  \begin{equation}
   \int_{\partial \Omega} \partial_t (n_j b_{ji} \psi_i) 
   =
   0
   \comma 
   t \in [0,T]
   ,
   \label{EQ205}
  \end{equation}
which is essential in the following construction. 

\subsection{Construction of a local-in-time solution}

\begin{theorem}
\label{T07}
Assume that $v_0 \in H^r$, where $r \in (2.5,3]$, satisfies the divergence and boundary conditions~\eqref{EQ199} and~\eqref{EQ202} at time $t = 0$. 
Then there exists a local-in-time solution $(v,q)$ to the system~\eqref{EQ199}--\eqref{EQ202} such that 
  \begin{align}
   \begin{split}
    &v \in L^\infty ([0,T]; H^r(\Omega))
    \\&
    v_t \in L^\infty([0,T]; H^{r-1}(\Omega))
    \\&
    q \in L^\infty ([0,T]; H^r(\Omega))
    ,
   \end{split}
   \llabel{EQ206}
  \end{align}
for some time $T > 0$ depending on $v_0$ and assumed data $\psi$ and~$a$. 
Moreover, the solution  $(v,q)$ satisfies the estimate 
  \begin{equation}
   \|v(t)\|_{H^r}
   + \|\nabla q \|_{H^{r-1}}
   \les
   P(
     \|v_0\|_{H^r},
     \|a\|_{H^{r+1}},
     \|\psi\|_{H^r}
     )
   + \int_0^t P(
     \|\psi\|_{H^r},
     \|\psi_t\|_{H^{r-1}},
     \|a\|_{H^{r+1}},
     \|a_t\|_{H^{r-1}}
     )
   \,ds
   ,
   \llabel{EQ207}
  \end{equation}
for $t \in [0,T]$. 
\end{theorem}

\begin{proof}
\emph{Step 1: The linearized problem.}
Following the method of~\cite{KuT}, we start by assuming higher regularity on the coefficients $\psi$ and~$a$. 
Assume that 
  \begin{equation}
   (\psi, \psi_t) \in L^\infty([0,T]; H^{r+1} \times H^r)
   \label{EQ208}
  \end{equation}
and 
  \begin{equation}
   (a,a_t) \in L^\infty([0,T]; H^{r+1} \times H^r)
   .
   \label{EQ209}
  \end{equation}
Next, we linearize~\eqref{EQ199} with the extension operator $E \colon H^k(\Omega) \to H^k(\mathbb{R}^3 )$. 
That is, we consider the transport equation
  \begin{equation}
   \partial_t v_i
   + E(\tilde{v}_m - \psi_m) E(a_{k m}) \partial_k v_i
   + E(a_{k i})E(\partial_k \tilde{q})
   = 0
   \inon{in $\mathbb{R}^3$}
   \comma i=1,2,3,
   \label{EQ210}
  \end{equation}
where $\tilde{v} \in L^\infty([0,T]; H^r)$ is given and $\tilde{q}$ is a solution to the elliptic Neumann problem,
  \begin{align}
   \begin{split}
    -\partial_j(b_{j i} a_{k i} \partial_k \tilde{q})
    &= 
    -\partial_j(\partial_t b_{j i} \tilde{v}_i)
    + b_{j i} \partial_j ((\tilde{v}_m - \psi_m)a_{k m}) \partial_k \tilde{v}_i
    \\& \indeq\indeq\indeq\indeq
    - (\tilde{v}_m - \psi_m)a_{k m} \partial_k b_{ji} \partial_j \tilde{v}_i
    + \mathcal{E}
    = \tilde{f}
    \inon{in $\Omega$ }
    ,
   \end{split}
   \label{EQ211}
  \end{align}
with boundary condition
  \begin{align}
   \begin{split}
    - n_j b_{j i} a_{k i} \partial_k \tilde{q}
    &=
    \partial_t (n_j b_{j i} \psi_i)
    - n_j \partial_t b_{j i} \tilde{v}_i
    \\& \indeq
    + (\tilde{v}_m - \psi_m) a_{k m} \partial_k^\tau (n_j b_{j i} \psi_i)
    - (\tilde{v}_m - \psi_m) a_{k m} \partial_k^\tau (n_j b_{j i}) \tilde{v}_i
    = \tilde{g}
    \inon{on $\partial\Omega$}
    .
   \end{split}
   \label{EQ212}
  \end{align}
The term $\mathcal{E}$ is defined by 
  \begin{align}
   \begin{split}
    \mathcal{E}(t) 
    &= 
    \frac{1}{|\Omega|} \int_{\partial \Omega} \partial_t (n_j b_{ji} \psi_i)
    - \frac{1}{|\Omega|} \int_{\partial\Omega} (\tilde{v}_m - \psi_m) a_{k m} \partial_k^\tau (n_j b_{j i}(\tilde{v}_i - \psi_i))
    \\&\indeq
    - \frac{1}{|\Omega|}\int_{\partial \Omega} (\tilde{v}_m - \psi_m) a_{k m}n_k n_jb_{j i} n_\ell \partial_\ell v_i
    + \frac{1}{|\Omega|} \int_\Omega (\tilde{v}_m - \psi_m)a_{k m} \partial_k(b_{ji} \partial_j \tilde{v}_i)
    .
   \llabel{EQ213}
   \end{split}
  \end{align}
Note that the first term of $\mathcal{E}$ does not vanish due to~\eqref{EQ208}--\eqref{EQ209}. 
The compatibility term is necessary and sufficient for the existence of a solution to the system~\eqref{EQ211}--\eqref{EQ212} as it ensures the compatibility 
  \begin{equation}
   \int_\Omega \tilde{f}
   =
   \int_{\partial \Omega} \tilde{g}
   .
   \llabel{EQ214}
  \end{equation}
A proof of compatibility is given in Lemma~\ref{L08}. 

The transport equation~\eqref{EQ210} with $v|_{t=0} = v_0$ admits a solution $v \in L_T^\infty H^r = L^\infty([0,T];H^r)$ for small time $T$ with the energy estimate
  \begin{equation}
   \|v\|_{L_T^\infty H^r} 
   \les
   \|v_0\|_{H^r}
   + \int_0^T P(
     \|\psi\|_{H^{r+1}},
     \|\psi_t\|_{H^r},
     \|a\|_{H^{r+1}},
     \|a_t\|_{H^{r}},
     \|\tilde{v}\|_{H^r}
     )
   \,ds
   .
   \label{EQ215}
  \end{equation}
Indeed, equation~\eqref{EQ210} is of the form
  \begin{equation}
   \partial_t v_i 
   + L_k \partial_k v_i
   =
   F_i
   \inon{in $\mathbb{R}^3$}
   \comma i=1,2,3.
   \llabel{EQ216}
  \end{equation}
Clearly, $L$ is bounded since we have assumed that the data belong to $L_T^\infty H^r$. 
To estimate $F$, elliptic regularity with the higher regularity assumption~\eqref{EQ208}--\eqref{EQ209} gives us 
  \begin{align}
   \begin{split}
    \|\nabla \tilde{q}\|_{H^r} 
    &\les
    \|\tilde{f}\|_{H^{r-1}(\Omega)}
    + \|\tilde{g}\|_{H^{r-1/2}(\partial \Omega)}
    \\&\leq
    P(
      \|\psi\|_{H^{r+1}},
      \|\psi_t\|_{H^r},
      \|a\|_{H^{r+1}},
      \|a_t\|_{H^r},
      \|\tilde{v}\|_{H^r}
    )
    .
   \end{split}
   \label{EQ217}
  \end{align}
After this, the derivation of the estimate~\eqref{EQ215} in $\mathbb{R}^3$ is standard. 

\emph{Step 2: Solution to the nonlinear problem under higher regularity.}
Through a contraction argument, we find a solution $(v,q)$ to~\eqref{EQ199} and~\eqref{EQ211}--\eqref{EQ212}.
Specifically, we use the iteration procedure
  \begin{equation}
   \partial_t v_i^{(n+1)}
   + E(v_m^{(n)} - \psi_m) E(a_{k m}) \partial_k v_i^{(n+1)}
   + E(a_{k i}) E(\partial_k q^{(n+1)})
   = 0
   ,
   \label{EQ218}
  \end{equation}
where  $q^{(n+1)}$ is a solution to the elliptic problem~\eqref{EQ211}--\eqref{EQ212} with $\tilde{v}$ replaced by~$v^{(n)}$.
Note that we get $q^{(1)}$ from $\tilde{v} = v_0$, and the rest is attained through the iteration. 
Also, the sequence $v^{(n)}$ is uniformly bounded for small time~$T$. To see this, a more precise estimate than~\eqref{EQ215} gives us 
  \begin{equation}
   \|v^{(n+1)}\|_{L_T^\infty H^r} 
   \les 
   \|v_0\|_{H^r} 
   + (\|v^{(n)}\|_{L_T^\infty H^r} +1 )
   \int_0^T P(
    \|\psi\|_{H^{r+1}},
    \|\psi_t\|_{H^r},
    \|a\|_{H^{r+1}},
    \|a_t\|_{H^r}
   )\, ds
   .
   \label{EQ219}
  \end{equation}
Hence, we may find $T$ small enough such that iterating~\eqref{EQ219} stays bounded as $n \to \infty$, that is, $\|v^{(n)}\|_{L_T^\infty H^r}\leq M$ for all~$n$.
From here on, we drop the subscript~$T$.
Once we show $v^{(n)} \mapsto v^{(n+1)}$ is a contraction in $L^\infty H^1$, we get $v^{(n)} \to v$ strongly in $L^\infty H^1$ by completeness. 
By compactness and uniqueness of weak-$\ast$ limits, however, we get $v \in L^\infty H^r$. 
A similar conclusion holds for $q^{(n)}$ through an elliptic estimate similar to~\eqref{EQ217}.

Now, we provide details for the contraction $v^{(n)} \mapsto v^{(n+1)}$ which is done in~$L^\infty H^1$.  
Define $V^{(n)} = v^{(n)} - v^{(n-1)}$ and $Q^{(n)} = q^{(n)} - q^{(n-1)}$. 
Taking a difference of equations, we find 
  \begin{align}
   \begin{split}
    &
    \partial_t V_i^{(n+1)} 
    + E(v_m^{(n)}) E(a_{k m}) \partial_k V_i^{(n+1)} 
    \\& \indeq\indeq\indeq\indeq 
    + E(V_m^{(n)}) E(a_{k m}) \partial_k v_i^{(n)}
    + E(a_{k i}) E(\partial_k Q^{(n+1)}) 
    = 0
    \inon{in $\mathbb{R}^3$}
    .
   \end{split} 
   \llabel{EQ220}
  \end{align}
Now, we apply $\Lambda = (1-\Delta)^{1/2}$ and test with $\Lambda V^{(n+1)}$ to get 
  \begin{align}
   \begin{split}
    \frac{1}{2}\partial_t \|\Lambda V^{(n+1)}\|_{L^2}^2
    &= 
    - \int \Lambda (E(v_m^{(n)}) E(a_{k m}) \partial_k V_i^{(n+1)}) \Lambda V_i^{(n+1)}
    \\&\indeq
    - \int \Lambda(E(V_m^{(n)})E(a_{k m})\partial_k v_i^{(n)}) \Lambda V_i^{(n+1)}
    - \int \Lambda (E(a_{k i}) E(\partial_k Q^{(n+1)}))\Lambda V_i^{(n+1)}
    \\&
    = I_1 + I_2 + I_3
    .
   \end{split}
   \llabel{EQ221}
  \end{align}
For $I_1$, commutator estimates and integration by parts give
  \begin{equation}
   |I_1| 
   \leq P(
     \|v^{(n)}\|_{H^r}, 
     \|a\|_{H^{r+1}}
     )
   \|V^{(n+1)}\|_{H^1}^2
   .
   \llabel{EQ222}
  \end{equation}
For $I_2$, applying Sobolev's inequality, we obtain
  \begin{align}
   \begin{split}
    |I_2|
    &\leq 
    \|E(V_m^{(n)}) E(a_{k m}) \partial_k v_i^{(n)}\|_{H^1} \|V^{(n+1)}\|_{H^1}
    \\&
    \leq P(
      \|v^{(n)}\|_{H^r},
      \|a\|_{H^{r+1}}
      )
    \|V^{(n+1)}\|_{H^1} \|V^{(n)}\|_{H^1} 
    , 
   \end{split}
   \llabel{EQ223}
  \end{align}
while for $I_3$, we have 
  \begin{equation}
   |I_3|
   \les
   \|a\|_{H^{r+1}} \|\nabla Q^{(n+1)}\|_{H^1} \|V^{(n+1)}\|_{H^1}
   .
   \llabel{EQ224}
  \end{equation}
Integrating and using the uniform bound on $\|v^{(n)}\|_{H^r}$, we get for $t \in [0,T]$,
  \begin{equation}
   \|V^{(n+1)}(t)\|_{H^1}^2
   \les 
   \int_0^t \|V^{(n+1)}\|_{H^1}^2 
   + \int_0^t \|V^{(n+1)}\|_{H^1}\|V^{(n)}\|_{H^1}
   + \int_0^t \|\nabla Q^{(n+1)}\|_{H^1} \|V^{(n+1)}\|_{H^1}
   ,
   \label{EQ225}
   \end{equation}
where the constant depends on $\psi$ and~$a$.
Next, we have the following elliptic estimate
  \begin{equation}
   \|\nabla Q^{(n+1)}\|_{H^1}
   \leq
   P(
     \|\psi\|_{H^{r+1}},
     \|\psi_t\|_{H^r},
     \|a\|_{H^{r+1}},
     \|a_t\|_{H^r}
   )
   \|V^{(n)}\|_{H^1}
   ,
   \label{EQ226}
  \end{equation}
which we can easily derive by taking a difference of equations and using elliptic regularity.
Hence, from~\eqref{EQ225}, we have from Young's inequality
  \begin{equation}
   \|V^{(n+1)}(t)\|_{H^1}^2 
   \les 
   \int_0^t \|V^{(n+1)}\|_{H^1}^2
   + \int_0^t \|V^{(n)}\|_{H^1}^2
   .
   \llabel{EQ227}
  \end{equation}
Using Gronwall's lemma, we get for small time $T$
  \begin{equation}
   \|V^{(n+1)}\|_{L^\infty H^1} 
   \leq 
   \frac{1}{2} \|V^{(n)}\|_{L^\infty H^1}
   .
   \llabel{EQ228}
  \end{equation}
Hence, $V^{n} \to 0$ strongly in $L^\infty H^1$ at an exponential rate so that $v^n \to v$ in $L^\infty H^1$ for some $v \in L^\infty H^1$. 

As pointed out above, this means $v \in L^\infty H^r$.
Using~\eqref{EQ226}, we know $\nabla q^{(n)}$ converges strongly in~$L^\infty H^1$.
Likewise, there exists $q \in L^\infty H^{r+1}$ defined up to a constant such that $\nabla q^{(n)} \to \nabla q$ in~$L^\infty H^1$. 
Integrating and taking the $H^1$ limit of~\eqref{EQ218} and restricting to $\Omega$ gives a solution $(v,q)$ to~\eqref{EQ199} and~\eqref{EQ211}--\eqref{EQ212}.

\emph{Step 3: Recovering boundary condition and an inhomogeneous divergence condition.}
As of now, we have a solution $(v,q)$ to the equation 
  \begin{equation}
   \partial_t v_i
   + (v_m - \psi_m)a_{k m} \partial_k v_i
   + a_{k i} \partial_k q
   = 0
   \inon{in $\Omega$}
   \comma i=1,2,3
   \label{EQ229}
  \end{equation}
and the elliptic Neumann problem 
\begin{align}
   \begin{split}
    -\partial_j(b_{j i} a_{k i} \partial_k q)
    &= 
    -\partial_j(\partial_t b_{j i} v_i)
    + b_{j i} \partial_j ((v_m - \psi_m)a_{k m}) \partial_k v_i
    \\& \indeq\indeq\indeq\indeq
    - (v_m - \psi_m) a_{k m} \partial_k b_{ji} \partial_j v_i
    + \mathcal{E}
    \inon{in $\Omega$ },
   \end{split} 
   \label{EQ230}
  \end{align}
with boundary condition
  \begin{align}
   \begin{split}
    -n_j b_{j i} a_{k i} \partial_k q
    &=
    \partial_t (n_j b_{j i} \psi_i)
    - n_j \partial_t b_{j i} v_i
    \\& \indeq
    + (v_m - \psi_m) a_{k m} \partial_k^\tau (n_j b_{j i} \psi_i)
    - (v_m - \psi_m) a_{k m} \partial_k^\tau (n_j b_{j i}) v_i
    \inon{on $\partial\Omega$}
    .
   \end{split}
   \label{EQ231}
  \end{align}
First, we recover the boundary condition $n_j b_{j i} (v_i - \psi_i) = 0$. 
Test~\eqref{EQ229} with $n_j b_{j i}$ and restrict to $\partial \Omega$ to get
  \begin{equation}
   n_j b_{j i} \partial_t v_i
   + (v_m - \psi_m)a_{k m} n_j b_{j i}\partial_k v_i
   + n_j b_{j i} a_{k i} \partial_k q
   = 0
   .
   \llabel{EQ232}
  \end{equation}
Rearranging some terms and using the decomposition
  \begin{equation}
   \partial_k 
   =
   (\delta_{k \ell} - n_k n_\ell) \partial_\ell
   + n_k n_\ell \partial_\ell
   = 
   \partial_k^\tau 
   + n_k n_\ell \partial_\ell
   ,
   \llabel{EQ232a}
  \end{equation}
we get
  \begin{align}
   \begin{split}
    \partial_t(n_j b_{j i} \partial_j v_i) 
    + (v_m - \psi_m)a_{k m} \partial_k^\tau (n_j b_{j i} v_i)
    &=
    -n_j b_{j i} a_{k i} \partial_k q
    + n_j \partial_t b_{j i} v_i
    + (v_m - \psi_m) a_{k m} \partial_k^\tau (n_j b_{j i}) v_i
    \\&\indeq
    - (v_m - \psi_m)a_{k m}n_k n_j b_{j i} n_\ell \partial_\ell v_i
    .
   \end{split}
   \label{EQ233}
  \end{align}
Substituting in~\eqref{EQ231}, this simplifies to 
  \begin{equation}
   \partial_t (n_j b_{j i} (v_i - \psi_i)) 
   + (v_m - \psi_m)a_{k m} \partial_k^\tau (n_j b_{j i} (v_i - \psi_i))
   = 
   - (v_m - \psi_m)a_{k m}n_k n_j b_{j i} n_\ell \partial_\ell v_i
   \inon{on $\partial \Omega $.}
   \label{EQ234}
  \end{equation}
Define $U = n_j b_{j i} (v_i - \psi_i)$ and note that the third term in~\eqref{EQ234} is actually $-J^{-1} U n_j b_{j i}n_\ell \partial_\ell v_i$.

Hence, we are interested in showing $U \equiv 0$ on $[0,T]$, where $U$ satisfies
  \begin{equation}
   \partial_t U
   + (v_m - \psi_m)a_{k m} \partial_k^\tau U
   = 
   - J^{-1} U n_j b_{j i}n_\ell \partial_\ell v_i
   \inon{on $\partial\Omega$.}
   \llabel{EQ237}
  \end{equation}
Testing against $U$ in $L^2(\partial \Omega)$, we have 
  \begin{equation}
   \frac{1}{2} \partial_t\|U\|_{L^2(\partial \Omega)}^2 
   =
   - \int_{\partial \Omega} (v_m - \psi_m) a_{k m} \partial_k^\tau U U
   - \int_{\partial\Omega} J^{-1} n_jb_{j i} n_\ell \partial_\ell v_i U^2
   .
   \llabel{EQ238}
  \end{equation}
Integrating by parts in the second term, we get
  \begin{equation}
   \frac{1}{2} \partial_t \|U\|_{L^2(\partial \Omega)}^2
   =
   \frac{1}{2}\int_{\partial \Omega} \partial_k^\tau ((v_m - \psi_m)a_{k m})U^2
   - \int_{\partial \Omega} J^{-1} n_jb_{j i} n_\ell \partial_\ell v_i U^2
   .
   \llabel{EQ239}
  \end{equation}
Since $H^r \hookrightarrow C^1(\bar{\Omega})$, we have 
  \begin{align}
   \begin{split}
    \partial_t \|U\|_{L^2(\partial \Omega)}^2 
    &\les
    \|\partial_k^\tau((v_m - \psi_m)a_{k m})\|_{L^\infty(\partial \Omega)} \|U\|_{L^2(\partial \Omega)}^2
    + \|J^{-1} n_jb_{j i} n_\ell \partial_\ell v_i\|_{L^\infty(\partial \Omega)} \|U\|_{L^2(\partial \Omega)}^2
    \\&
    \les
    P(\|v\|_{H^r},
    \|\psi\|_{H^r},
    \|a\|_{H^{r+1}}
    )
    \|U\|_{L^2(\partial \Omega)}^2
    .
   \end{split}
   \llabel{EQ240}
  \end{align}
Since $U(0) = 0$, applying Gronwall gives us $U = 0$ for all time $t \in [0,T]$.

Now, we derive an inhomogeneous divergence condition
  \begin{equation}
   b_{ji} \partial_j v_i 
   =
   \frac{1}{|\Omega|} \int_{\partial \Omega} n_j b_{j i} \psi_i
   .
   \llabel{EQ240b}
  \end{equation}
We apply $b_{j i} \partial_j$ to~\eqref{EQ229} and substitute~\eqref{EQ230} to get
  \begin{equation}
   \partial_t (b_{j i} \partial_j v_i)
   + (v_m - \psi_m)a_{k m} \partial_k (b_{j i} \partial_j v_i)
   = 
   \mathcal{E}
   \inon{in $\Omega$}
   .
   \label{EQ241}
  \end{equation}
Since $n_j b_{j i}(v_i - \psi_i) = 0$, the compatibility term $\mathcal{E}(t)$ becomes
  \begin{equation}
   \mathcal{E}(t)
   =
   \frac{1}{|\Omega|} \int_{\partial \Omega} \partial_t (n_j b_{j i} \psi_i) 
   + \frac{1}{|\Omega|} \int_{\Omega} (v_m - \psi_m) a_{k m} \partial_k (b_{j i} \partial_j v_i)
   .
   \llabel{EQ242}
  \end{equation}
Hence, defining $\mathcal{D} = b_{j i} \partial_j v_i$, the equation~\eqref{EQ241} becomes
  \begin{equation}
   \partial_t \mathcal{D} 
   + \mathcal{A} \cdot \nabla \mathcal{D}
   = 
   \frac{1}{|\Omega|} \int_{\partial \Omega} \partial_t (n_j b_{ji} \psi_i)
   +\frac{1}{|\Omega|} \int_\Omega (\mathcal{A} \cdot \nabla \mathcal{D})
   ,
   \label{EQ243}
  \end{equation}
where we define $\mathcal{A}_k = (v_m - \psi_m)a_{k m}$. 
Performing an $\dot{H}^1$ estimate, we get $\nabla \mathcal{D} = 0$. 
Revisiting~\eqref{EQ243}, we obtain
  \begin{equation}
   \mathcal{D}(t) 
   =
   \frac{1}{|\Omega|} \int_{\partial \Omega} n_j b_{j i} \psi_i
   ,
   \llabel{EQ243b}
  \end{equation}
where we note that $\int_{\partial \Omega} n_jb_{ji}(0) \psi_i(0) = 0$ since $n_j b_{ji} v_i = n_j b_{ji} \psi_i$ and $\mathcal{D}|_{t=0} = 0$.

\emph{Step 4: Energy estimate for $v$ under original regularity assumptions.}
In order to bound $v$ against the data $\psi$ and $a$ under the original regularity assumptions~\eqref{EQ203}--\eqref{EQ204}, we apply the a-priori bounds from Section~\ref{sec02}. However, the solution $(v,q)$ we found above is not smooth so we need to use difference quotients to justify this. 
Define 
  \begin{equation}
   D_{h,\ell}f(x) 
   =
   \frac{f(x+he_\ell) - f(x)}{h}
   ,
   \llabel{EQ244}
  \end{equation}
for $x \in \mathbb{R}^3$, $\ell = 1,2,3$ and $h \in \mathbb{R}^3 \setminus \{0\}$.
For the remainder of the justification, we write $D = D_{h,\ell}$.

The only part in our estimates in Section~\ref{sec02} which requires this modification is Lemma~\ref{L05}. Conveniently, $q \in H^{r+1}$, which is already smooth enough to justify integration by parts for the pressure estimates in Lemma~\ref{L02}.

As for $v$, we first apply the variable curl $\eps_{ijk} b_{\ell j} \partial_\ell (\cdot)_k$ to~\eqref{EQ199} to get 
  \begin{equation}
   \partial_t \zeta_i 
   + (v_m - \psi_m)a_{rm}\partial_r \zeta_i
   = 
   \zeta_p a_{m p}\partial_m v_i 
   + f_i
   \comma i = 1,2,3,
   \inon{in $\Omega$}
   ,
   \llabel{EQ245}
  \end{equation}
  where $\zeta_i = \eps_{ijk} b_{\ell j} \partial_\ell v_k$ and $f$ is defined in~\eqref{EQ60}.
By Lemma~\ref{L04}, we can extend to $\mathbb{R}^3$ and perform estimates on the problem
  \begin{equation}
   \bar{J} \partial_t \theta_i 
   + (\tilde{v}_m - \psi_m) \tilde{b}_{rm} \partial_r \theta_i
   =
   \theta_p \tilde{b}_{m p} \partial_m \tilde{v}_i
   + \bar{J} \tilde{f}_i
   \inon{in $\mathbb{R}^3$}
   \comma i=1,2,3.
   \label{EQ246}
  \end{equation}
It then suffices to reprove Lemma~\ref{L05} for 
  \begin{equation}
   Y
   =
   \int_{\mathbb{R}^3} \bar{J}|\Lambda^{r-2}D\theta|^2
   .
   \llabel{EQ247}
  \end{equation}
Applying $\Lambda^{r-2}D$ to~\eqref{EQ246}, testing with $\Lambda^{r-2}D\theta$, and integrating, we get a similar equation to~\eqref{EQ71},
  \begin{align}
   \begin{split}
    \frac{1}{2} \frac{dY}{dt} 
    &= 
    -\int \Big(
      \Lambda^{r-2}D((\tilde{v}_m - \tilde{\psi}_m)\tilde{b}_{r m} \partial_r \theta_i) 
      - (\tilde{v}_m - \tilde{\psi}_m)\tilde{b}_{r m} \Lambda^{r-2}D\partial_r \theta_i
    \Big) 
    \Lambda^{r-2}D\theta_i
    \\&\indeq
    - \frac{1}{2}\int (\tilde{v}_m - \tilde{\psi}_m)\tilde{b}_{r m} \partial_r (\Lambda^{r-2}D\theta_i)^2
    \\&\indeq
    + \int \Big(
      \Lambda^{r-2}D(\theta_p \tilde{b}_{m p} \partial_m \tilde{v}_i)
      - \theta_p \tilde{b}_{m p} \Lambda^{r-2}D\partial_m \tilde{v}_i
    \Big)
    \Lambda^{r-2}D\theta_i
    \\&\indeq
    + \int \theta_p \tilde{b}_{m p}\partial_m \Lambda^{r-2}D\tilde{v}_i \Lambda^{r-2}D\theta_i
    \\&\indeq
    - \int \Big(
      \Lambda^{r-2}D(\bar{J}\partial_t \theta_i)
      - \bar{J} \Lambda^{r-2}D(\partial_t \theta_i)
    \Big)
    \Lambda^{r-2}D\theta_i
    \\&\indeq
    + \frac{1}{2}\int \partial_t \bar{J}|\Lambda^{r-2}D\theta |^2
    + \int \Lambda^{r-2}D(\bar{J} \tilde{f}) \Lambda^{r-2}D\theta_i
    \\&=
    I_1 + \cdots + I_7
    .
   \end{split}
   \label{EQ248}
  \end{align}
Indeed, with only $v \in H^r$, we may integrate by parts in the second line of~\eqref{EQ248}.
Otherwise, using the product rule for difference quotients, we simplify the right-hand side enough such that $D$ falls directly on $\theta$ whenever possible.
After this, the estimates proceed as they did in Lemma~\ref{L05}

We then apply Lemma~\ref{L06} to get the estimate
  \begin{equation}
  \|v(t)\|_{H^r}^2 
   \les
   P(
     \|v_0\|_{H^r},
     \|\psi\|_{H^r},
     \|a\|_{H^{r+1}}
     )
   + \int_0^t P(
     \|\psi\|_{H^r},
     \|\psi_t\|_{H^{r-1}},
     \|a\|_{H^{r+1}},
     \|a_t\|_{H^{r-1}}
   )
   \, ds
   ,
   \label{EQ249}
  \end{equation}
for $t \in [0,T]$ where $T$ depends on $v_0$, $\psi$ and~$a$.
Note the bound is at the level of the original regularity assumptions~\eqref{EQ203}--\eqref{EQ204}.

\emph{Step 5: Solution to the nonlinear problem with original regularity assumptions.}
Now, we derive our solution $(v,q)$ through an approximation argument. 
Let
  \begin{align}
   \begin{split}
    (\psi,\psi_t) &\in L^\infty([0,T];H^r \times H^{r-1})
    ,
    \\
    (\psi^{(n)}, \psi_t^{(n)}) &\in L^\infty ([0,T];H^{r+1}\times H^r)
   \end{split}
   \llabel{EQ250}
  \end{align}
and
  \begin{align}
   \begin{split}
    (a,a_t) &\in L^\infty([0,T]; H^{r+1}\times H^{r-1})
    ,
    \\
    (a^{(n)},a_t^{(n)}) &\in L^\infty([0,T]; H^{r+1} \times H^r)
   \end{split}
   \llabel{EQ251}
  \end{align}
be such that
  \begin{equation}
   (\psi^{(n)}, \psi_t^{(n)}) \to (\psi, \psi_t) 
   \inon{in $L^\infty([0,T]; H^r \times H^{r-1})$}
   \label{EQ290}
  \end{equation}
and 
  \begin{equation}
   (a^{(n)}, a_t^{(n)}) \to (a,a_t)
   \inon{in $L^\infty([0,T]; H^{r+1} \times H^{r-1})$.}
   \label{EQ291}
  \end{equation}
By the previous steps, for each $n$ there exists a solution $(v^{(n)},q^{(n)}) \in L^\infty([0,T];H^r \times H^{r+1})$ to~\eqref{EQ199}--\eqref{EQ202} with $q $ defined up to a constant. 
Due to~\eqref{EQ290}--\eqref{EQ291} and the estimate~\eqref{EQ249}, we have a uniform bound on $v^{(n)}$ in $H^r$ against the respective weaker norms as well as on $\nabla q^{(n)}$ in $H^{r-1}$ through the elliptic estimate~\eqref{EQ14}.
Since $q^{(n)}$ is defined up to a constant, we are free to normalize it by assuming  $\int_\Omega q^{(n)} = 0$ to get $q^{(n)}$ uniformly bounded in~$H^r$. 
By Lemma~\ref{L07}, we know $v_t^{(n)}$ is uniformly bounded in $L^\infty([0,T];H^{r-1})$.
By compactness, we may pass to a subsequence, denoted the same, such that 
  \begin{align}
   \begin{split}
    &v^{(n)} \to v
    \inon{weakly-$\ast$ in $L^\infty([0,T];H^r)$}
    ,
    \\&
    q^{(n)} \to q
    \inon{weakly-$\ast$ in $L^\infty([0,T];H^r)$}
    ,
    \\&
    v_t^{(n)} \to v_t
    \inon{weakly-$\ast$ in $L^\infty([0,T];H^{r-1})$}
    .
   \end{split}
   \llabel{EQ292}
  \end{align}
By the Aubin-Lions lemma (see~\cite{S}), we actually get $v^{(n)} \to v$ strongly in $C([0,T];H^{s})$ for any $s \in [r-1,r)$.

We now pass to the limit in the Euler equation~\eqref{EQ199}. 
Let $\phi \in C_0^\infty(\Omega \times (0,T))$. 
By weak-$\ast$ convergence
  \begin{equation}
   \langle v_t^n - v, \phi \rangle \to 0 
   .
   \llabel{EQ293}
  \end{equation}
For the nonlinear term, we have
  \begin{align}
   \begin{split}
    &|\langle v_m^{(n)} a_{k m}^{(n)}\partial_k v_i^{(n)} - v_m a_{k i} \partial_k v_i, \phi\rangle|
    \\&\indeq \les
    \|v^{(n)} - v\|_{L^\infty H^{r-1}}\|a^{(n)}\|_{L^\infty H^{r+1}} \|\nabla v^{(n)}\|_{L^\infty H^{r}} \|\phi\|_{L^1L^2}
    \\&\indeq 
    + \|v\|_{L^\infty H^{r}}\|a^{(n)} - a\|_{L^\infty H^{r+1}} \|\nabla v^{(n)}\|_{L^\infty H^{r}} \|\phi\|_{L^1 L^2}
    \\&\indeq
    + \|v\|_{L^\infty H^{r}} \|a\|_{L^\infty H^{r+1}} \|\nabla v^{(n)}-\nabla v\|_{L^\infty H^{r-2}} \|\phi\|_{L^1 L^2}
    .
   \end{split}
   \llabel{EQ295}
  \end{align}
From the strong convergence for $v$ and $a$, the right-hand side tends to zero. 
A similar computation gives 
  \begin{equation}
   |\langle \psi_m^{(n)}a_{km}^{(n)} \partial_k v_i^{(n)} - \psi_m a_{k m} \partial_k v_i, \phi\rangle| \to 0
   .
   \llabel{EQ296}
  \end{equation}
For the pressure term, weak-$\ast$ convergence a similar computation as above gives us
  \begin{equation}
   \langle a_{k i}^{(n)} \partial_k q^{(n)} - a_{k i} \partial_k q, \phi\rangle \to 0
   .
   \llabel{EQ294}
  \end{equation}
Next, we look at the divergence condition. 
Indeed,
  \begin{align}
   \begin{split}
    |\langle a_{j i}^{(n)} \partial_j v_i^{(n)} - a_{j i} \partial_j v_i, \phi\rangle|
    &\les
    \|a^{(n)} - a\|_{L^\infty H^{r+1}} \|\nabla v^{(n)}\|_{L^\infty H^{r-1}} \|\phi\|_{L^1L^2}
    \\&\indeq
    + \|a\|_{L^\infty H^{r+1}} \|\nabla v^{(n)} - \nabla v\|_{L^\infty H^{r-2}}\|\phi\|_{L^1L^2}
    ,
    \label{EQ297}
   \end{split}
  \end{align}
which tends to zero.
Recall, however, that
  \begin{equation}
   a_{ji}^{(n)}\partial_j v_i^{(n)} 
   =
   (J^{(n)})^{-1} \frac{1}{|\Omega|} \int_{\partial \Omega}n_j b_{j i}^{(n)} \psi_i^{(n)}
   \rightarrow
   J^{-1} \frac{1}{|\Omega|} \int_{\partial \Omega} n_j b_{j i} \psi_i
   ,
   \label{EQ297b}
  \end{equation}
and the compatibility condition~\eqref{EQ205} guarantees the right-hand-side of~\eqref{EQ297b} to vanish. 
Hence, $a_{ji}\partial_j v_i = 0$.

Finally, we verify the boundary condition. 
Let $\rho \in C_0^\infty(\partial \Omega \times (0,T))$. 
See that 
  \begin{align}
   \begin{split}
    |\langle n_j a_{ji}^n v_i^{(n)} - n_j a_{ji} v_i, \rho\rangle_{\partial \Omega}|
    &\les 
    \|a^{(n)} - a\|_{L^\infty H^{r+1}}\|v^{(n)}\|_{L^\infty H^r}\|\phi\|_{L^1L^2(\partial \Omega)}
    \\&\indeq
    + \| a \|_{L^\infty H^{r+1} }\|v^{(n)} - v\|_{L^\infty H^{r-1}}\|\phi\|_{L^1L^2(\partial\Omega)}
    ,
    \llabel{EQ298}
   \end{split}
  \end{align}
which tends to zero. 
A similar computation gives 
  \begin{equation}
   |\langle n_j a_{j i}^{(n)} \psi_i^{(n)} - n_j a_{j i} \psi_i, \rho\rangle_{\partial \Omega} | \to 0
   ,
   \llabel{EQ299}
  \end{equation}
which gives $n_ja_{j i} (v_i - \psi_i) = 0$.

Hence, $(v,q) \in L^\infty([0,T];H^r \times H^r)$ solves~\eqref{EQ199}--\eqref{EQ202}.  
Since $\phi$ and $\rho$ are arbitrary and $v,q$ are of sufficiently high regularity, the system is satisfied almost everywhere in space and time.  
\end{proof}

\subsection{Verification of compatibility}
In this section, we verify that the elliptic system~\eqref{EQ210}--\eqref{EQ211} is compatible.
We should note that the earlier Neumann problem~\eqref{EQ17} and~\eqref{EQ20} is clearly compatible by its construction.
Recall the system~\eqref{EQ210}--\eqref{EQ211},
  \begin{align}
   \begin{split}
    -\partial_j(b_{j i} a_{k i} \partial_k \tilde{q})
    &= 
    -\partial_j(\partial_t b_{j i} \tilde{v}_i)
    + b_{j i} \partial_j ((\tilde{v}_m - \psi_m)a_{k m}) \partial_k \tilde{v}_i
    \\& \indeq\indeq\indeq\indeq
    - (\tilde{v}_m - \psi_m)a_{k m} \partial_k b_{ji} \partial_j \tilde{v}_i
    + \mathcal{E}
    = \tilde{f}
    \inon{in $\Omega$ }
    ,
   \end{split}
   \label{EQ300}
  \end{align}
with boundary condition
  \begin{align}
   \begin{split}
    - n_j b_{j i} a_{k i} \partial_k \tilde{q}
    &=
    \partial_t (n_j b_{j i} \psi_i)
    - n_j \partial_t b_{j i} \tilde{v}_i
    \\& \indeq
    + (\tilde{v}_m - \psi_m) a_{k m} \partial_k^\tau (n_j b_{j i} \psi_i)
    - (\tilde{v}_m - \psi_m) a_{k m} \partial_k^\tau (n_j b_{j i}) \tilde{v}_i
    = \tilde{g}
    \inon{on $\partial\Omega$}
   \end{split}
   \label{EQ301}
  \end{align}
with 
  \begin{align}
   \begin{split}
    \mathcal{E}(t) 
    &= 
    \frac{1}{|\Omega|}\int_{\partial\Omega} \partial_t (n_j b_{j i} \psi_i) 
    - \frac{1}{|\Omega|} \int_{\partial\Omega} (\tilde{v}_m - \psi_m) a_{k m} \partial_k^\tau (n_j b_{j i}(\tilde{v}_i - \psi_i))
    \\&\indeq
    - \frac{1}{|\Omega|}\int_{\partial \Omega} (\tilde{v}_m - \psi_m) a_{k m}n_k n_jb_{j i} n_\ell \partial_\ell v_i
    + \frac{1}{|\Omega|} \int_\Omega (\tilde{v}_m - \psi_m)a_{k m} \partial_k(b_{ji} \partial_j \tilde{v}_i)
    .
   \end{split}
   \llabel{EQ302}
  \end{align}
\begin{lemma}
\label{L08}
The system~\eqref{EQ300}--\eqref{EQ301} satisfies 
  \begin{equation}
   \int_\Omega \tilde{f}
   =
   \int_{\partial \Omega} \tilde{g}
   .
   \llabel{EQ303}
  \end{equation}
\end{lemma}
\begin{proof}
Integrating $\tilde{f}$ over $\Omega$, we have 
  \begin{align}
   \begin{split}
    \int_\Omega \tilde{f}
    &= 
    -\int_\Omega \partial_j(\partial_t b_{j i} \tilde{v}_i)
    + \int_\Omega \partial_j(b_{j i} (\tilde{v}_m - \psi_m)a_{km}) \partial_k \tilde{v}_i
    - \int_\Omega (\tilde{v}_m - \psi_m)a_{k m} \partial_k b_{j i} \partial_j \tilde{v}_i
    \\&\indeq
    + \int_{\partial\Omega} \partial_t (n_j b_{j i} \psi_i)
    - \int_{\partial \Omega} (\tilde{v}_m - \psi_m)a_{k m} \partial_k^\tau (n_j b_{j i} (\tilde{v}_i - \psi_i))
    \\&\indeq
    - \int_{\partial \Omega} (\tilde{v}_m - \psi_m) a_{k m}n_k n_jb_{j i} n_\ell \partial_\ell v_i
    + \int_\Omega (\tilde{v}_m - \psi_m)a_{k m} \partial_k (b_{j i} \partial_j \tilde{v}_i)
    ,
   \end{split}
   \label{EQ304}
  \end{align}
where we used Piola in the third term. 
Focusing on the third term, we integrate by parts in the $x_j$ variable, decompose $\partial_k v_i$ into tangential and normal derivatives on the boundary, and use the product rule,
  \begin{align}
   \begin{split}
    \int_{\Omega} \partial_j (b_{j i}(\tilde{v}_m - \psi_m)a_{k m})\partial_k \tilde{v}_i 
    &= 
    \int_{\partial \Omega} (\tilde{v}_m - \psi_m) a_{k m} n_j b_{j i} \partial_k \tilde{v}_i
    - \int_\Omega (\tilde{v}_m - \psi_m)a_{k m} b_{j i} \partial_k \partial_j \tilde{v}_i
    \\&=
    \int_{\partial\Omega} (\tilde{v}_m - \psi_m) a_{k m} \partial_k^\tau (n_j b_{j i} \tilde{v}_i)
    - \int_{\partial \Omega} (\tilde{v}_m - \psi_m) a_{k m} \partial_k^\tau (n_j b_{j i}) \tilde{v}_i
    \\&\indeq
    + \int_{\partial \Omega} (\tilde{v}_m - \psi_m) a_{k m} n_k n_j b_{j i} n_\ell \partial_\ell v_i
    - \int_\Omega (\tilde{v}_m - \psi_m)a_{k m} \partial_k (b_{j i} \partial_j \tilde{v}_i)
    \\&\indeq
    + \int_\Omega (\tilde{v}_m - \psi_m) a_{k m} \partial_kb_{j i} \partial_j \tilde{v}_i
    .
   \end{split}
   \llabel{EQ305}
  \end{align}
Substituting this into~\eqref{EQ304}, some cancellations and the divergence theorem give
  \begin{align}
   \begin{split}
    \int_\Omega \tilde{f}
    &= 
    -\int_{\partial \Omega} n_j\partial_t b_{j i} \tilde{v}_i
    + \int_{\partial\Omega} \partial_t (n_j b_{j i} \psi_i)
    \\&\indeq
    + \int_{\partial \Omega} (\tilde{v}_m - \psi_m) a_{k m} \partial_k^\tau (n_j b_{j i} \psi_i)
    - \int_{\Omega} (\tilde{v}_m - \psi_m)a_{k m} \partial_k^\tau (n_j b_{j i})\tilde{v}_i
    .
   \end{split}
   \llabel{EQ306}
  \end{align}
The right-hand side is exactly $\int_{\partial\Omega} \tilde{g}$, so the proof is complete. 
\end{proof}

\begin{remark}
\label{R01}
Note that the compatibility condition~\eqref{EQ205} is necessary for any solution to the system~\eqref{EQ199}--\eqref{EQ200}. 
To see this, write
  \begin{equation}
   \int_{\partial \Omega} n_j b_{j i} \psi_i 
   =
   \int_{\partial\Omega} n_j b_{j i} v_i
   =
   \int_{\partial \Omega}b_{j i} \partial_j v_i
   = 0
   ,
   \label{EQ307}
  \end{equation}
where we used the boundary and divergence conditions along with the divergence theorem. 
\end{remark}

\startnewsection{Beale-Kato-Majda Criterion}{sec04}
Finally, we consider the breakdown of solutions in $H^r$ and show that it solely depends on the $\BMO$ norm of the variable vorticity $\zeta_i = \eps_{ijk}b_{\ell j} \partial_\ell v_k$ and the $H^1$ norm of the velocity~$v$.
In particular, we provide a theorem in the case $r=3$ and when the coefficients $a,\psi$ are taken with some higher regularity.
For ease of notation, we keep $r$ in place of $3$ in the following computations. 

For $r \in (2.5,3)$, the matter is more complicated due to the method of extension in acquiring estimates on~$\|v\|_{H^r}$.
We will not treat that case here.  

\begin{theorem}
\label{T09}
Let $v$ be a solution to~\eqref{EQ06}--\eqref{EQ08} in the class $C([0,T]; H^r)$ where $r = 3$. 
Suppose that $T = \hat{T}$ is the first time such that $v \notin C([0,T]; H^r(\Omega))$ with
  \begin{equation}
   (\psi,\psi_t) \in L^\infty([0,\infty); H^{r+1} \times H^r)
   \llabel{EQ401}
  \end{equation}
and 
  \begin{equation}
   (a,a_t) \in L^\infty([0,\infty); H^{r+2} \times H^r)
   .
   \llabel{EQ402}
  \end{equation}
Then 
  \begin{equation}
   \int_0^{\hat{T}} (\|v(t)\|_{H^1} + \|\zeta(t)\|_{\BMO})\, dt
   =
   \infty
   .
   \llabel{EQ403}
  \end{equation}
\end{theorem}

Before the proof of the main theorem, we provide a preliminary elliptic estimate.  
We note the following lemma achieves this as well with $r\in (2.5,3)$ and the original regularity assumptions on $\psi$ and $a$ in~\eqref{EQ01}--\eqref{EQ02}. 
\begin{lemma}
\label{L10}
With the same hypotheses as in Theorem~\ref{T09}, we have 
  \begin{equation}
   \|v\|_{\BMO}
   + \|\nabla v\|_{\BMO}
   \les
   \|v\|_{H^1}
   + \|\zeta\|_{\BMO}
   .
   \llabel{EQ404}
  \end{equation}
\end{lemma}
\begin{proof}
We provide an elliptic estimate for the system given by~\eqref{EQ89} and~\eqref{EQ99}. 
Split up $v_j = u_j + w_j$, where $w_j$ is a solution to the Neumann problem
  \begin{align}
   \begin{split}
    -\partial_m (b_{m k} b_{\ell k} \partial_\ell w_j)
    &=
    \partial_m (b_{m k} \varphi_{kj})
    \\
    -n_m b_{m k} b_{\ell k} \partial_\ell v_j 
    &= 
    n_m b_{m k} \varphi_{kj}
    ,
   \end{split}
   \llabel{EQ405}
  \end{align}
with $\varphi$ defined in~\eqref{EQ84}. 
The remaining $u_j$ is a solution to a Neumann problem containing the lower-order terms in $v$,
  \begin{align}
   \begin{split}
    -\partial_m (b_{m k} b_{\ell k} \partial_\ell u_j)
    &=
    -\partial_\ell (b_{m k} \partial_m b_{\ell j} v_k)
    -\partial_\ell b_{m k} \partial_m (b_{\ell k} v_k)
    \\
    -n_m b_{m k} b_{\ell k} \partial_\ell u
    &= 
    F_j|_{\partial \Omega}
    ,
   \end{split}
   \llabel{EQ26}
  \end{align}
where $\nu \in C^\infty(\bar{\Omega})$ extends the outward normal vector $n$, and 
  \begin{equation}
   F_j
   =
   b_{\ell j} \partial_\ell(v_k b_{m k} \nu_m)
   - \partial_\ell (b_{m k} \nu_m) b_{\ell j} v_k
   .
   \llabel{EQ406}
  \end{equation}
Note that $j = 1,2,3$ and all other indices are summed  except that $k$ is summed in $\{1,2,3\}\setminus \{j\}$. 
For $w_j$, we obtain
  \begin{equation}
    \|\nabla w\|_{\BMO}
    \les
    \|v\|_{H^1}
    + \|\zeta\|_{\BMO}
    ,
    \llabel{EQ407}
  \end{equation}
which can be found in Troianiello~\cite[Theorem 3.16(ii)]{T}. 
For $u_j$, we have the $L^3$ estimate,
  \begin{equation}
   \|D^2 u\|_{L^3} 
   \les
   \|v\|_{W^{1,3}}
   .
   \llabel{EQ408}
  \end{equation}
These two estimates together with $v = u + w$ give 
  \begin{equation}
   \|\nabla v\|_{\BMO}
   \les
   \|\nabla u\|_{\BMO}
   + \|\nabla w\|_{\BMO}
   \les 
   \|v\|_{H^1} 
   + \|\zeta\|_{\BMO}
   + \|v\|_{W^{1,3}}
   ,
   \llabel{EQ409}
  \end{equation}
where we used the embedding $\|f\|_{\BMO} \les \|\nabla f\|_{L^3}$.
Interpolation between $L^2$ and $\BMO$ along with the same embedding leaves us with 
  \begin{equation}
   \|\nabla v \|_{\BMO}
   \les 
   \|v\|_{H^1} + \|\zeta\|_{\BMO}
   .
   \llabel{EQ410}
  \end{equation}
Using embedding and interpolation once more, we get 
  \begin{equation}
   \|v\|_{\BMO}
   + \|\nabla v\|_{\BMO}
   \les
   \|v\|_{H^1} 
   + \|\zeta\|_{\BMO}
   .
   \llabel{EQ411}
  \end{equation}
\end{proof}

\begin{proof}[Proof of Theorem~\ref{T09}]
Towards a contradiction, we assume
  \begin{equation}
   \int_0^{\hat{T}} (
     \|v(t)\|_{H^1} 
     + \|\zeta(t)\|_{\BMO}
   )
   \, dt
   <
   \infty
   .
   \llabel{EQ412}
  \end{equation}
We first provide an estimate for a solution $v(t) \in H^{r+1}$ and proceed with an approximation argument. 

Assume that $v(t) \in H^{r+1}$ is a solution to 
  \begin{equation}
   \partial_t v_i
   + (v_m - \psi_m) a_{k m} \partial_k v_i
   + a_{k i} \partial_k q
   =
   0
   .
   \label{EQ413}
  \end{equation}
For $|\alpha| \leq 3$, we apply  $D^\alpha$ and test with~$D^\alpha v_i$. 
After integrating by parts and using the boundary condition $n_k a_{k m} (v_m - \psi_m) = 0$, we have
  \begin{align}
   \begin{split}
    \frac{1}{2}\partial_t \int |D^\alpha v|^2
    &= 
    \frac{1}{2}\int \partial_k((v_m - \psi_m) a_{k m}) |D^\alpha v|^2
    - \int D^\alpha (a_{ki} \partial_k q) D^\alpha v_i    
    \\&\indeq
    - \int \big\{ 
            D^\alpha[(v_m - \psi_m)a_{k m}\partial_k v_i] 
            - (v_m - \psi_m)a_{k m} D^\alpha \partial_k v_i
    \big\} D^\alpha v_i
    .
   \end{split} 
   \llabel{EQ414}
  \end{align}
Applying a commutator estimate, we get 
  \begin{equation}
   \partial_t \|D^\alpha v\|_{L^2}^2
   \les
   \|v\|_{W^{1,\infty}} \|v\|_{H^r}^2
   + \|\nabla q \|_{H^r}\|v\|_{H^r}
   .
   \llabel{EQ415}
  \end{equation}
The pressure term satisfies the elliptic Neumann problem given in~\eqref{EQ46} and~\eqref{EQ48}. 
Elliptic regularity and the Leibniz rule then give us
  \begin{equation}
   \|\nabla q\|_{H^r} 
   \les
   \|v\|_{W^{1,\infty}}\|v\|_{H^r}
   ,
   \llabel{EQ416}
  \end{equation}
so that 
  \begin{equation}
   \partial_t \|v\|_{H^r}^2 
   \les 
   \|v\|_{W^{1,\infty}} \|v\|_{H^r}^2
   .
   \label{EQ417}
  \end{equation}
Applying Gronwall's lemma, we have
  \begin{equation}
   \|v(t)\|_{H^r}
   \les
   \|v_0\|_{H^r} \exp \Bigg(
        \int_0^T \|v(t)\|_{W^{1,\infty}}\, dt
   \Bigg)
   \comma
   t \in [0,T]
   ,
   \label{EQ418}
  \end{equation}
with $T \leq \hat{T}$. 
We can control $\|v(t)\|_{W^{1,\infty}}$ using the bound (see~\cite{KT}),
  \begin{equation}
   \|f\|_{L^\infty}
   \les 
   1 + \|f\|_{\BMO} (1+ \log^+(\|f\|_{W^{s,p}}))
   ,
   \llabel{EQ419}
  \end{equation}
which holds for any $f \in W^{s,p}$ with $s > n/p$ and $p \in (1,\infty)$. 
Applying this to $v$ and $\nabla v$ and using the trivial bound $\|v\|_{H^{r-1}} \leq \|v\|_{H^r}$, Lemma~\ref{L10} gives us 
  \begin{equation}
          \|v\|_{W^{1,\infty}} 
   \les 
   1 + (\|v\|_{H^1} + \|\zeta\|_{\BMO})(1 + \log^+(\|v\|_{H^r}))
   .
   \label{EQ420}
  \end{equation}
Next, we substitute~\eqref{EQ420} into~\eqref{EQ418} to get 
  \begin{equation}
   \log^+(\|v\|_{H^r})
   \les
   F(T)
   + \int_0^T L(s) \log^+(\|v\|_{H^r})\, ds
   ,
   \label{EQ421}
  \end{equation}
where we shorthanded 
  \begin{equation}
   L(t)
   = 
   \|v(t)\|_{H^1} + \|\zeta(t)\|_{\BMO}
   \llabel{EQ423}
  \end{equation}
and 
  \begin{equation}
   F(T)
   = 
   \log(\|v_0\|_{H^r})
   + T
   + \int_0^T L(s)\, ds
   .
   \llabel{EQ422b}
  \end{equation}
Finally, we apply Gronwall's lemma to~\eqref{EQ421} and obtain
  \begin{equation}
   \|v(t)\|_{H^r}
   \les
   \exp\Bigg( 
           F(\hat{T}) 
           + \int_0^{\hat{T}} F(s) L(s) \exp \Bigg( 
                   \int_0^{\hat{T}} L(\tau) \, d\tau
                   \Bigg)
           \Bigg)
   =:
   K
   .
   \label{EQ423}
  \end{equation}

This bound holds for a solution $v(t) \in H^{r+1}$. 
However, we can show that~\eqref{EQ418} holds for a solution $v(t) \in H^r$ using an approximation argument. 
After verifying this, the rest of the above argument up to~\eqref{EQ423} follows as before. 

Let $v_0 \in H^r$ satisfy the divergence and boundary conditions at $t=0$. 
So Theorem~\ref{T02} gives us a solution $v \in C([0,T_1];H^r)$, where $T_1$ depends only on $2K$ and the coefficients $a$ and $\psi$ which are bounded uniformly past $\hat{T}$ in their respective norms. 
Next, we can find a sequence $v_0^n \in H^{r+1}$ such that $v_0^n \to v_0$ in $H^r$ and which satisfies the divergence and boundary conditions at $t =0$. 
The existence of such a sequence is ensured by mollification and the continuity of a projection operator $P_a$ defined by 
  \begin{equation}
   (P_a f)_i 
   = 
   f_i - a_{k i} \partial_k \phi
   ,
   \llabel{EQ424}
  \end{equation}
where $\phi$ solves the elliptic Neumann problem
  \begin{align}
    \begin{split}
   \partial_j(b_{ji} a_{ki} \partial_k \phi)
   &= 
   \partial_j (b_{ji} f_i)
   \\
   n_j b_{j i} a_{k i} \partial_k \phi
   &=
   n_j b_{j i} (f_i - \psi_i)
   .
  \end{split}
   \llabel{EQ425}
  \end{align}
Hence, we get a sequence of solutions $v^n \in C([0,T_n];H^{r+1})$. 
In particular, we can take the time of existence to be $T_1$ for all $n$ sufficiently large. 
To see this, first note that $\|v_0\|_{H^r} \leq 2K$, which is immediate by inspecting the right-hand-side of~\eqref{EQ423}. 
By $H^r$ convergence, $v_0^n$ enjoys the same bound for large $n$, and by repeating an identical estimate to get~\eqref{EQ417}, we can arrive at 
  \begin{equation}
   \partial_t \|v\|_{H^{r+1}}
   \les
   \|v\|_{H^r}
   \|v\|_{H^{r+1}}
   .
   \llabel{EQ426}
  \end{equation}
As in, the time of blow-up in the $H^r$ norm must forerun that of the $H^{r+1}$ norm. 
So we can opt for a smaller time of existence than $T^n$ which only depends on $\|v_0^n\|_{H^r}$ which are bounded uniformly by $2K$ for large $n$, allowing us to soften to~$T_1$. 

Next, we can repeat the estimates above to get 
  \begin{equation}
   \|v^n(t)\|_{H^r}
   \les 
   \|v_0^n\|_{H^r} \exp \Bigg(
           \int_0^{T_1} \|v^n(s)\|_{W^{1,\infty}}\, ds
   \Bigg)
   \comma 
   t \in [0,T_1]
   .
   \label{EQ427}
  \end{equation}
Now, we pass to the limit in~\eqref{EQ427}. 
For each $t \in [0,T_1]$, we know that $\|v^n(t)\|_{H^r}$ is uniformly bounded and hence there exists a subsequence, denoted the same, such that $v^n(t) \weakto \tilde{v}(t)$ weakly in~$H^r$. 
Moreover,  $\tilde{v}(t) = v(t)$ and $v^n(t) \to v(t)$ in $H^s$ for any $s < r$, for which it suffices to show $v^n(t) \to v(t)$ in~$L^2$. 

Indeed, define $w^n = v^n - v$.
Then $n_m b_{m k}w_k^n = 0$ and $b_{ji} \partial_j w_i^n = 0$. 
Multiplying~\eqref{EQ413} by $J$ and taking a difference gives us 
  \begin{equation}
   J\partial_t w_i^n 
   + (v_m^n - \psi_m) b_{k m} \partial_k w_i^n
   + w_m^n b_{m k} \partial_k v_i
   + b_{k i} \partial_k (q^n - q)
   = 0
   .
   \llabel{EQ428}
  \end{equation}
We test with $w_i^n$ and obtain
  \begin{align}
   \begin{split}
    \frac{1}{2} \partial_t \int J |w^n|^2
    &= 
    \frac{1}{2} \int \partial_t J |w^n|^2
    - \frac{1}{2} \int b_{k m} (v_m^n - \psi_m) \partial_k|w^n|^2
    \\&\indeq
    - \int w_m^n b_{k m} \partial_k v_i w_i^n
    - \int \partial_k (b_{k i} (q^n - q)) w_i^n
    .
   \end{split}
   \llabel{EQ429}
  \end{align}
After integrating by parts, all boundary terms and pressure terms vanish to leave us with
  \begin{equation}
   \frac{1}{2} \partial_t \int J |w^n|^2
   = 
   -\frac{1}{2} \int \partial_t J |w^n|^2
   + \frac{1}{2}\int b_{k m} \partial_k \psi_m |w^n|^2
   - \int w_m^n b_{k m} \partial_k v_i w_i^n
   .
   \llabel{EQ430}
  \end{equation}
Since $J$ is uniformly bounded above and below by positive constants, we get
  \begin{equation}
   \|w^n\|_{L^2}
   \les
   \|w_0^n(t)\|_{L^2} \exp\Bigg(
           \int_0^{T_1} \|\nabla v \|_{L^\infty}
   \Bigg)\, ds
   \comma 
   t \in [0,T_1]
   .
   \llabel{EQ431}
  \end{equation}
Hence, $v^n(t) \to v(t)$ in $L^2$ for $t \in [0,T_1]$.
The convergence in $H^s$ for $s < r$ follows by interpolation and boundedness of $v^n(t)$ in~$H^r$.
By embedding, this means $v^n(t) \to v(t)$ in~$W^{1,\infty}$. 
Passing to the limit in~\eqref{EQ427}, we arrive at~\eqref{EQ418} for $v(t) \in H^r$ with $T=T_1$. 
Proceeding as we did after~\eqref{EQ418}, we get $\|v(t)\|_{H^r} \leq 2K$ for $t \in [0,T_1]$.

Taking a new initial value as $v(T_1)$ and repeating this process extends our solution and its desired control to $[0,2T_1]$ since $\|v(T_1)\|_{H^r} \leq 2K$. 
Because $K$ is unchanged for $T < \hat{T}$, we can iterate this process a finite number of times until we extend our solution past $\hat{T}$, revealing a contradiction. 
\end{proof}

\section*{Acknowledgments}
\rm
BI and IK were supported in part by the NSF grant DMS-2205493.


\end{document}